\documentclass[11pt,a4paper]{article}
\usepackage[UKenglish]{babel}
\usepackage[titletoc,title]{appendix}
\usepackage{amsmath,amsfonts,amssymb,bm,amsthm}
\usepackage[colorlinks,citecolor=blue,linkcolor=blue,pagebackref]{hyperref}
\usepackage{tikz, esint}
\usetikzlibrary{shapes,arrows}
\usetikzlibrary{intersections,shapes.arrows}
\usepackage{caption}
\numberwithin{equation}{section}
\newtheorem{theorem}{Theorem}[section]
\newtheorem{lemma}[theorem]{Lemma}
\newtheorem{corollary}[theorem]{Corollary}
\newtheorem{proposition}[theorem]{Proposition}
\theoremstyle{definition}

\newtheorem{definition}[theorem]{Definition}
\newtheorem{assumption}[theorem]{Assumption}
\theoremstyle{remark}
\newtheorem{remark}[theorem]{Remark}

\usepackage{nicefrac}

\newcommand{\R}{\mathbb{R}}
\newcommand{\Z}{\mathbb{Z}}
\newcommand{\N}{\mathbb{N}}
\newcommand{\E}{\mathbb{E}}
\newcommand{\A}{\mathcal{A}}
\newcommand{\B}{\mathcal{B}}

\newcommand{\M}{\mathcal{M}}
\newcommand{\Pro}{\mathbb{P}}

\renewcommand{\tilde}{\widetilde}

\renewcommand{\epsilon}{\varepsilon}

\newcommand{\loc}{\mathrm{loc}}

\newcommand{\dr}{\mathrm{d}}
\newcommand{\ds}{\mathrm{d}s\, }
\newcommand{\dt}{\mathrm{d}t\, }
\renewcommand{\hat}{\widehat}

\newcommand{\e}{\varepsilon}

\newcommand{\osc}{{\rm osc}}
\newcommand{\ext}{{\rm ext}}
\renewcommand{\hom}{{\rm hom}}

\newcommand{\bbb}{\color{black}}

\newcommand{\igor}{\color{black}}

\usepackage[T1]{fontenc}
\usepackage[top=1in, bottom=1in, left=1in, right=1in]{geometry}
\usepackage{fancyhdr}
\usepackage{enumerate}
\usepackage{xcolor,cancel}
\usepackage{relsize}

\usepackage{multicol}
\usepackage{longtable}

\usepackage{orcidlink}

\allowdisplaybreaks

\usepackage{pifont}
\renewcommand*{\backrefalt}[4]{%
\ifcase #1 %
No citations%
\or
\ding{43}~p.~#2%
\else
\ding{43}~pp.~#2%
\fi}

\begin{document}

\title{Quantitative estimates for high-contrast random media}
\author{
Peter Bella\,\orcidlink{0000-0002-1660-1711}\thanks{TU Dortmund, 
Fakult\"at f\"ur Mathematik, Vogelpothsweg 87, 44227 Dortmund, Germany;
\text{peter.bella@udo.edu},
\url{https://www.mathematik.tu-dortmund.de/sites/bella}.
}
\and
Matteo Capoferri\,\orcidlink{0000-0001-6226-1407}\thanks{Dipartimento di Matematica ``Federigo Enriques'', Università degli Studi di Milano, Via C.~Saldini 50, 20133 Milano, Italy \emph{and} Maxwell Institute for Mathematical Sciences, Edinburgh \&
Department of Mathematics,
Heriot-Watt University,
Edinburgh EH14 4AS,
UK;
\text{matteo.capoferri@unimi.it},
\url{https://mcapoferri.com}.
}
\and
Mikhail Cherdantsev\,\orcidlink{0000-0002-5175-5767}\thanks{
School of Mathematics,
Cardiff University,
Senghennydd Road,
Cardiff CF24~4AG,
UK;
\text{CherdantsevM@cardiff.ac.uk}.
}
\and
Igor Vel\v{c}i\'c\,\orcidlink{0000-0003-2494-2230}\thanks{
Faculty of Electrical Engineering and Computing, 
University of Zagreb, Unska 3, 
10000 Zagreb, Croatia;
\text{igor.velcic@fer.hr}.
}}


\date{2 June 2025}

\maketitle

\vspace{-.5cm}

\begin{abstract}
This paper studies quantitative homogenization of elliptic equations with random, uniformly elliptic coefficients that vanish in a union of random holes. Assuming an upper bound on the size of the holes and a separation condition between them, we derive optimal bounds for the regularity radius $r_*$
and suboptimal growth estimates for the corrector. These results are key ingredients for error analysis in stochastic homogenization and serve as crucial input for recent developments in the double-porosity model, such as those by Bonhomme, Duerinckx, and Gloria (\url{https://arxiv.org/abs/2502.02847}). 

\


\

{\bf 2020 MSC classes: }
primary 
35J70, 
60H25; 
secondary 
35B27,  
35B40,  
74A40, 
74Q05. 
\end{abstract}

\tableofcontents

    \newpage

\allowdisplaybreaks

\section{Introduction}

Many materials found in nature exhibit material parameters that vary on very small scales, which presents significant challenges when it comes to simulating their behaviour. This difficulty arises from the pronounced separation of scales between the material heterogeneity and the overall size of the specimen being studied. The concept of homogenization offers a potential solution to this problem by transforming what might be perceived as a disadvantage into an advantage. By approximating highly oscillatory problems with simpler homogeneous models, one can facilitate more efficient treatment through various methods, such as numerical simulations.

In particular, while the mathematical theory of homogenization is considerably more straightforward for cases involving periodic microstructures, most real-world materials do not conform to this idealised structure. Instead, they often possess random characteristics or, at best, their structural properties are understood only in statistical terms. As a result, there is a growing need for an appropriate framework that addresses these complexities; hence, the corresponding theory of stochastic homogenization becomes essential. This theory aims to provide insights and methodologies for effectively analysing materials whose properties are inherently random or statistically defined, thereby enhancing our understanding and predictive capabilities regarding their behaviour under various conditions.

\smallskip

One of the most fundamental and accessible problems in the study of partial differential equations (PDEs) is the Poisson equation and its heterogeneous counterpart, expressed in the form of uniformly elliptic equations. A common formulation for such homogenization problem is given by  
\begin{equation}\label{eqint_1}
    -\nabla \cdot (a(x/\varepsilon) \nabla u_\varepsilon) = f\, ,
\end{equation}  
where $\varepsilon > 0$ represents a small parameter characterising the scale of heterogeneity. The matrix-valued function $a(x/\varepsilon)$ encodes fine-scale material properties, such as electrical or thermal conductivity, which may vary rapidly at a microscopic level. 

The study of such equations is motivated by numerous applications in physics, engineering, and materials science where multiscale phenomena naturally arise. As $\varepsilon \searrow 0$, corresponding to an increasingly finer microscopic structure, it has been rigorously shown that the solutions $(u_\varepsilon)$ to \eqref{eqint_1} converge to the solution $u_{\mathrm{hom}}$ of a homogenised problem:  
\begin{equation}
    -\nabla \cdot (a_{\mathrm{hom}} \nabla u) = f\, .
\end{equation}  
Here $a_{\mathrm{hom}}$ is a constant matrix encoding effective or ``averaged'' material properties at the macroscopic scale. These results hold under various assumptions on $a$, such as periodicity or stochastic ergodicity, different types of right-hand sides $f$, and in different functional norms. 

Although these \emph{qualitative} results --- asserting that $u_\varepsilon \to u_{\mathrm{hom}}$ as $\varepsilon \to 0^+$ --- are theoretically important, they are insufficient for practical purposes. In real-world scenarios, $\varepsilon$ is small but finite; thus, \emph{quantitative} estimates become essential. Specifically, one seeks precise error bounds on $\|u_\varepsilon - u_{\mathrm{hom}}\|$ that depend explicitly on $\varepsilon$. Such estimates provide valuable insights into how well the homogenized solution approximates its heterogeneous counterpart for given values of $\varepsilon$. Moreover, obtaining these bounds requires careful mathematical analysis that often depends on additional structural properties of $a(x)$.

An equally crucial aspect involves efficiently computing the homogenized coefficients $a_{\mathrm{hom}}$. While theoretical frameworks like periodic homogenization or stochastic homogenization provide formal definitions for $a_{\mathrm{hom}}$, their numerical computation can be challenging due to high-dimensional integrals or complex probabilistic structures. Developing robust algorithms to approximate $a_{\mathrm{hom}}$ accurately remains an active area of research with significant implications for computational multiscale modeling.


\smallskip


To understand the behaviour of heterogeneous media as $\varepsilon \searrow 0$, a common approach in stochastic homogenization is to instead fix $\varepsilon = 1$ and study the problem on increasingly larger scales. This perspective avoids directly shrinking $\varepsilon$ and instead focuses on how solutions behave over growing domains --- an approach we will also pursue here.

A central object in the quantitative theory of stochastic homogenization is the \emph{corrector} $\phi$, which serves as a correction to linear functions so that they become $a$-harmonic. Mathematically, the corrector is defined via 
the equation
\begin{equation}\label{eqint_corr}
    -\nabla \cdot (a \nabla (x_i + \phi_i)) = 0,
\end{equation}  
where $x_i$ represents a coordinate direction. The corrector encodes key information about the microscopic oscillations of the coefficient field $a(x)$ and plays a critical role in determining both homogenization errors and other properties of solutions, such as regularity. 

One of its most important features is its growth --- or, rather, its lack thereof --- which reflects sublinear behavior at large scales. This sublinearity is pivotal for understanding quantitative error estimates in homogenization as well as for establishing regularity results for solutions to heterogeneous PDEs. In particular, quantifying this sublinearity forms the cornerstone of modern quantitative theories of stochastic homogenization (see, e.g., works by Gloria and Otto~\cite{gloria2012quantitative}). 

More precisely, one studies the size (or decay) of $\phi/R$ over balls of growing radii $R$. A common approach to measure the latter involves bounding quantities like  
$\inf_c \frac{1}{R^2} \fint_{B_R} |\phi - c|^2$, 
which captures normalised deviations from constants within a ball $B_R$. The smallness of this quantity has profound implications: for example, it connects directly to large-scale regularity properties such as Schauder estimates or simpler but practically useful Lipschitz estimates. For instance, Lipschitz bounds take the form  
$\fint_{B_r} |\nabla u|^2 \lesssim C \fint_{B_R} |\nabla u|^2$, 
for all large enough radii $r \leq R$, where $u$ is an $a$-harmonic function. Such estimates are fundamental tools in controlling oscillatory behaviors at different scales.

Interestingly, these Lipschitz-type bounds can also be leveraged through sensitivity analysis to gain control over the growth behavior of the corrector itself. By carefully analysing how changes propagate across scales, one can link these bounds back to sublinearity properties of $\phi$. Buckling is then used to make this ``circular argument'' between corrector growth and regularity of solutions mathematically rigorous. 

\smallskip

The double porosity problem concerns elliptic equations with coefficients that are of order $\e^2$ in soft inclusions. These inclusions have size $\e$ and are spaced at a distance $\e$, leading to resonances between the microscopic structure and the governing partial differential equation. Such resonances give rise to intriguing properties of materials modeled by these equations, as the interplay between the geometry and the small-scale coefficients significantly influences their macroscopic behavior, most notably their spectral properties, see e.g. \cite{CCV1}, \cite{CCV2} and \cite{systems} for some recent results in this area.

The effective behaviour of such problem can be described using standard homogenization techniques. Specifically, it involves constructing a corrector for a degenerate elliptic problem where the coefficient vanishes inside the perforations (soft inclusions). This degeneracy reflects the physical reality of highly permeable or weakly resistive regions within the material. These PDEs can be equivalently formulated as uniformly elliptic equations in perforated domain with zero Neumann boundary conditions on the perforations. The resulting homogenised model captures both the influence of these soft inclusions and their interactions with the surrounding medium, providing insights into macroscopic transport properties such as diffusion or conductivity.

This framework is particularly useful for studying materials with hierarchical porous structures or composite media where soft inclusions play a dominant role in determining effective behaviour.

\smallskip

The study of homogenization in perforated domains has a long and rich history, primarily focusing on the case with zero Dirichlet boundary conditions imposed on the holes. This line of research was initiated by Tartar~\cite{tartar}, and further developed by Cioranescu and Murat, who introduced the oscillating test function method~\cite{CioranescuMurat}. Using this approach, they analysed the homogenization of the Poisson equation in domains perforated with periodically arranged critically sized holes. In the limit, they derived an effective equation featuring an additional ``strange'' Brinkman zero-th order term, which is defined through the capacity of the holes.

A similar result was obtained independently by Papanicolaou and Varadhan for randomly placed critically sized holes under the assumption that their distribution obeys a stationary and ergodic law~\cite{PapaVaradhan}. Their work demonstrated that stochastic methods can also be employed effectively to handle such problems in random settings.

Subsequent research revealed that these techniques are robust enough to address more complex equations in perforated domains. For instance, Allaire extended these ideas to study incompressible Navier-Stokes equations modeling fluid dynamics in perforated geometries~\cite{Allaire}. More recently, progress has been made on compressible fluids models (see, e.g., \cite{BellaFeireislOschmann} for periodic as well as~\cite{BellaOschmann} for random configurations of holes) or effective viscosity problem where the obstacles move with the flow~\cite{MitiaGloria23}.

\smallskip

Turning back to the elliptic equations, let us now briefly discuss the strategy to quantify the growth of the corrector, building on the theory of stochastic homogenization for uniformly elliptic equations developed over the past decade by Otto and coauthors. This framework provides powerful tools to analyse sublinear growth properties of correctors in random environments.

The (lack of) growth of the corrector stems from cancellations in its gradient. Thus, understanding how quickly the gradient decorrelates is key. Since the gradient of the corrector $\nabla \phi$ solves an equation with coefficients given by $a$, this requires quantifying the ergodicity of randomness in $a$. In this context, coefficient fields $a$ are treated as elements of a probability space $\Omega$. To achieve this quantification, functional inequalities come into play. These inequalities --- extensions of classical Poincar\'e and Sobolev inequalities to infinite-dimensional probability spaces --- are known as Spectral Gap and logarithmic Sobolev inequalities (log-Sobolev). 

Unlike standard Poincar\'e inequalities on $\mathbb{R}^d$, where gradients appear on the right-hand side (RHS), in Spectral Gap inequalities, these gradients are replaced by ``vertical'' derivatives. Vertical derivatives measure oscillations or sensitivity of random variables when values of $a$ (as elements in the probability space $\Omega$) are altered locally within different parts of the domain. One significant advantage of this approach is that it decouples analytical arguments (e.g., PDE estimates) from stochastic considerations (e.g., ergodicity). Another advantage is its generality: such functional inequalities can be independently verified for many stochastic models~\cite{DG1,DG2}, making them widely applicable.

To control oscillations appearing on the RHS of these functional inequalities, we need to understand how much $\nabla \phi$ changes if we modify $a$ locally within a ball $B_r(x)$. Denoting this modified coefficient field by $a'$, we study differences between $\nabla \phi$ and $\nabla \phi'$, where $\phi'$ is the corrector associated with $a'$. Observing that this difference satisfies an elliptic equation with a RHS supported in $B_r(x)$ allows us to efficiently apply energy estimates. Combined with large-scale Lipschitz bounds --- which propagate control over $ |\nabla u|^2 $ across scales --- this leads to bounds on oscillations. These bounds then feed into functional inequalities to control variances of large scale averages of $\nabla \phi$.

When dealing with equations involving random coefficients, it is important to note that large-scale Lipschitz estimates do not hold uniformly across all scales but instead depend on a random scale $ r_*(x) $. This scale $ r_*(x) $, which varies spatially due to randomness, marks the threshold where large-scale regularity begins to hold around a point $x$. Crucially, $r_*(x)$ is closely linked to the sublinearity properties of correctors over balls $B_r(x)$ for $r \geq r_*(x)$ --- an observation highlighted in~\cite{mjm}. This connection plays a pivotal role in bridging sublinearity results for correctors with quantitative homogenization theory. Specifically, understanding how $r_*(x)$ behaves provides insight into how solutions transition from microscopic oscillations to macroscopic regularity.

Another key idea involves the use of massive correctors, denoted by $\phi_T$, for $T \gg 1$. These are defined as whole-space solutions to $\frac{1}{T} \phi_T - \nabla \cdot a (\nabla \phi_T +\mathbf{e}) = 0$,
where $ \mathbf{e} \in \mathbb{R}^d $ represents a fixed direction. Massive correctors serve as an effective proxy for the standard corrector $ \phi $. On one hand, the massive term modifies the dependence of $ \phi_T $ on the coefficient field $ a(x) $, effectively localising its sensitivity to regions of size approximately $O(\sqrt{T})$. On the other hand, any power-law growth observed in $ |\phi_T| / R^\alpha$ with $T=R^2$ directly translates into analogous growth behavior for the original corrector $ |\phi| / R^\alpha$.

\smallskip

Before the work of Otto and coauthors, Armstrong and Smart~\cite{ArmstrongSmart} developed quantitative estimates using a completely different technique, building on the foundational work of Avellaneda and Lin~\cite{AvellanedaLin}. They also utilised the variational structure (initially for symmetric coefficient fields $ a $, later extended to the general case in~\cite{ArmstrongKuusiMourrat16}), analysing the averaged energy over an increasing sequence of triadic cubes. By combining this with a convex dual quantity, they established quantitative convergence rates to the homogenised limit. In all these works the regularity theory plays a crucial role, also enabling a precise understanding of the behavior of solutions~\cite{ArmstrongKuusiMourrat17}. For further details on this framework and its applications, we refer to the comprehensive book by Armstrong, Kuusi, and Mourrat~\cite{ArmstrongKuusiMourrat19}.

More recently, homogenization techniques have been applied in settings where classical homogenization does not occur. For example, consider Brownian motion in two spatial dimensions with drift along the level lines of a Gaussian Free Field (GFF). Physicists predicted that such motion would exhibit superdiffusive behavior with logarithmic corrections—a phenomenon supported by stochastic methods (see e.g.,~\cite{Balint}). However, these earlier works did not identify the precise rate of superdiffusion.

While the presence of superlinear diffusion indicates a failure of homogenization, Chatzigeorgiou et al.~\cite{Chatz++} truncated correlations of the field on scale $L$ and employed homogenization methods to analyse diffusion speed. By averaging over all random realisations of the field, they were able to let $L \to \infty$ to obtain the correct superlinear rate, including identifying the appropriate prefactor in the limit of vanishing drift. 
Using local quantities and homogenization techniques across multiple scales, combined with a careful analysis of interactions between these scales, Armstrong, Bou-Rabee, and Kuusi~\cite{ArmstrongBuuKuusi} achieved stronger results in several ways: their findings are quenched (i.e., valid for almost every realization of the Gaussian Free Field), identify the optimal constant without requiring weak drift, and are presented in the form of a scaling limit. 
These advancements represent significant progress in understanding anomalous diffusion phenomena in random environments. Furthermore, it is believed that these techniques can be adapted to handle perturbatively supercritical fields. In such cases, instead of logarithmic corrections leading to superlinear growth rates, one would expect power-law behaviors—an exciting direction for future research.


\smallskip



A homogenization setting closely related to the present situation is that of non-uniformly elliptic equations. This area has been motivated by studies of the invariance principle for random walks in degenerate random environments~\cite{BS1,BS2} or on quantitative homogenization in supercritical percolation clusters~\cite{ArmstrongDario, Dario}. Quantitative homogenization of degenerate equations was considered by the first author and Kniely~\cite{bella_kniely} (see also~\cite{ArmstrongKuusi, ClozeauGloriaQi}), building on earlier work concerning regularity theory in such degenerate setting~\cite{BFO}. These studies provide key tools to address challenges posed by degeneracy, where coefficients may be close to vanish or become unbounded.

\

Shortly before the release of this manuscript, an article by Bonhomme, Duerinckx, and Gloria~\cite{BonDueGlo} appeared online as a preprint.
At first glance, their work appears to tackle very similar issues; however, upon closer inspection, it becomes evident that while there is some overlap between their and the present work, their results are largely complementary to ours. Specifically, in~\cite{BonDueGlo}, the focus lies on establishing qualitative homogenization results as well as quantitative statements under \emph{the assumption} that correctors are stationary objects in $L^2(\Omega)$. 
As discussed earlier in this introduction, controlling correctors is a fundamental prerequisite for achieving quantitative homogenization results. One of the main contributions of our work lies precisely in providing such control within this framework. Thus, while both studies address related problems, they do so from different perspectives and with distinct focal points.

\subsection*{Structure of the paper}
\addcontentsline{toc}{subsection}{Structure of the paper}

Our paper is structured as follows.

In Section~\ref{Statement of the problem} we introduce our geometric and probabilistic setting, set out our assumption and formulate our mathematical model. This section also contains definition for most of the notation used throughout the paper.

Section~\ref{Main result} concisely states our main results.

The proof of our most important result, Theorem~\ref{theorem 2}, is postponed until the end of the next Section~\ref{Proof of main theorem}, which comprises a series of propositions building up to said proof.

The paper is complemented by an appendix, Appendix~\ref{Appendix hole filling}, providing the proof of a version of the Hole Filling Lemma for perforated domains.

Finally, for the reader's convenience, we summarise in the table below the notation most frequently used throughout the paper.

\subsection*{List of notation}
\addcontentsline{toc}{subsection}{List of notation}
\begin{longtable}{l l}
\hline
\\ [-1em]
\multicolumn{1}{c}{\textbf{Symbol}} & 
  \multicolumn{1}{c}{\textbf{Description}} \\ \\ [-1em]
 \hline \hline \\ [-1em]
 $(\,\cdot\,)_{*t}$ & Convolution with a Gaussian of variance $t$ \\ \\ [-1em]
$\partial^\osc_{Q}(\,\cdot\,)$ & Oscillation~\eqref{oscillation} \\ \\ [-1em]
$\square_L(x)$ & Hypercube of side $L$ centred at $x\in \R^d$ \\ \\ [-1em]
$a_\pm$ & Ellipticity constants --- see~\eqref{coefficients 2}\\ \\ [-1em]
$a_\omega$ & Matrix-valued random coefficient field --- see~\eqref{coefficients 1}, \eqref{coefficients 2}\\ \\ [-1em]
$|A|$ & Lebesgue measure of $A\subset \R^d$ \\ \\ [-1em]
$\overline{A}$ & Closure of $A\subset \R^d$  \\ \\ [-1em]
$B_r(x)$ (resp.~$B_r$) & Open ball of radius $r>0$ centred at $x\in \R^d$ (resp.~at the origin) \\ \\ [-1em]
$\tilde B_r$ & Lipschitz domain, $B_r \subset\tilde B_r \subset B_{r+1}$ --- see statement of Proposition~\ref{proposition 1} \\ \\ [-1em]
$\mathcal{B}_\omega^k$ & Extension domain of $\omega^k$ --- see~Remark~\ref{remark geometric assumptions}(b) \\ \\ [-1em]
$\beta$ & Constant in the Multiscale Spectral Gap --- see Assumption~\ref{multiscale spectral gap} \\ \\ [-1em]
$d$ & Spatial dimension, $d\ge 2$ \\ \\ [-1em]
$\mathcal{D}(\A)$ & Domain of the operator $\A$ \\ \\ [-1em]
$\E[\,\cdot\,]$ ($\E_L[\,\cdot\,]$)  & Expectation (expectation w.r.t.~ensemble scaled by $L$ and coefficients $a(L\,\cdot\,)$)\\ \\ [-1em]
$\mathrm{Exc}(\nabla u;B_r)$  & Excess~\eqref{definition Exc} \\ \\ [-1em]
$(\,\cdot\,)^\ext$  & Extension of $(\,\cdot\,)$ as per Theorem~\ref{theorem extension} \\ \\ [-1em]
$f(x,\omega):=\overline{f}(T_x\omega)$ & Stationary extension of the random variable $\overline{f}$ \\ \\ [-1em]
$g_T$ & Vector field, solution of~\eqref{equation definition of g_T} \\ \\ [-1em]
$I$ & Identity matrix \\ \\ [-1em]
$\M=\M_\omega$ & Matrix, i.e.~the complement of $\overline{\omega}$ \\ \\ [-1em]
$q$ & Flux~\eqref{definition q} \\ \\ [-1em]
$q_T$ & Massive flux~\eqref{definition qT} \\ \\ [-1em]
$r_*$ & Regularity radius~\eqref{definition r_*} \\ \\ [-1em]
$r_{**}$ & Massive regularity radius --- see~\eqref{proposition 2 equation 1} \\ \\ [-1em]
$\sigma=(\sigma_{ijk})$ & Flux corrector --- see Definition~\ref{definition flux corrector} \\ \\ [-1em]
$\sigma(\A)$ & Spectrum of the operator $\A$ \\ \\ [-1em]
$T_y$ & Dynamical system \eqref{Ty} on $\Omega$ acting by translation \\ \\ [-1em]
$\phi$ & Homogenisation corrector --- see~Definition~\ref{definition corrector} \\ \\ [-1em]
$\phi_T$ & Massive corrector --- see~Definition~\ref{definition massive corrector} \\ \\ [-1em]
$\pi(l)$ & Weight in the Multiscale Spectral Gap --- see~Assumption~\ref{multiscale spectral gap} \\ \\ [-1em]
$\mathrm{Var}[\,\cdot\,]$ ($\mathrm{Var}_L[\,\cdot\,]$)  & Variance (variance w.r.t.~ensemble scaled by $L$ and coefficients $a(L\,\cdot\,)$)\\ \\ [-1em]
$\omega$ & Collection of inclusions ($\omega\subset \R^d$, $\omega\in \Omega$) --- see~Assumption~\ref{assumption 1} \\ \\ [-1em]
$\omega^k$ & Individual inclusion, connected component of $\omega$ --- see~Assumption~\ref{assumption 1} \\ \\ [-1em]
$\omega_T(x)$ & Exponential weight~\eqref{exponential weight} \\ \\ [-1em]
$\chi_A$ & Characteristic function of the set $A$ \\ \\ [-1em]
$(\Omega, \mathcal{F}, \mathbb{P})$ & Probability space \\ \\ [-1em]
\hline
\end{longtable}

\section{Statement of the problem}
\label{Statement of the problem}

\subsection{Geometric and probabilistic setting}

Working in Euclidean space $\R^d$, $d\ge 2$, equipped with the Lebesgue measure, we denote by $|A|$ the Lebesgue measure of a measurable set $A\subset \R^d$, by $\overline{A}$ its closure, and by $\chi_A$ its characteristic function. Furthermore, we denote by $B_r(x)$ the (open) ball of radius $r>0$ centred at $x\in \mathbb{R}^d$, and by $\square_L(x):=x+[-\nicefrac{L}{2},\nicefrac{L}{2}]^d$ the (closed) cube centred at $x$ of sidelength $L>0$. We shall often drop the argument and simply write $B_r$ (resp.~$\square_L$) when the ball (resp.~cube) is centred at the origin $x=0$.

By $C(\mu_1, \dots, \mu_n)$ we denote a positive constant depending only on the parameters $\mu_1, \dots, \mu_n$. When we write $C$ with no arguments we imply that $C$ is a universal constant.

Given a function $f\in L^1(A)$, we denote by 
\[
\fint_A f\,\dr x:=\frac{1}{|A|}\int f(x)\,\dr x
\]
its mean over the measurable set $A$.

\

We define our probability space $(\Omega,\mathcal{F}, \mathbb{P})$ as follows. 

\

We define $\Omega$ to be the set of all admissible collections of random \emph{inclusions}, namely, disconnected subsets $\omega\subset\R^d$ satisfying the following geometric assumptions.

\begin{definition}[Minimally smooth set {\cite[Definition~2.1]{decay}\cite[Chapter~VI, Sect.~3.3]{stein}}]
\label{def:minimally_smooth}
An open set $A \subset \R^d$ is said to be {\it minimally smooth with constants $(N,C_1,C_2)$}
if there exists a countable (possibly finite) family of open sets $\{U_i\}_{i\in I}$ such that
	\begin{enumerate}
		\item[(a)]  each $x\in\R^d$ is contained in at most $N\in \N$ of the open sets $U_i$;
		\item[(b)] for every $x\in\partial A$ there exists $i\in I$ such that $B_{C_1}(x)\subset U_i$;
		\item[(c)] for every $i\in I$ the set $\partial A\cap U_i$ is,  in a suitably chosen coordinate system,  the graph of a Lipschitz function with Lipschitz seminorm not exceeding $C_2$.
	\end{enumerate}
\end{definition}

\begin{assumption}
\label{assumption 1}
The set $\mathbb{R}^d\setminus \omega$ is connected for all $\omega\in\Omega$. Furthermore, $\omega$ can be written as the disjoint union
\[
\omega=\coprod_{k\in \mathbb{N}}\omega^k
\]
of sets $\omega^k$ --- referred to as \emph{inclusions} --- such that\footnote{Here and further on, $\operatorname{diam}A:=\sup_{x,y\in A}|x-y|$ denotes the (Euclidean) diameter of $A$ whereas $\operatorname{dist}(A,B):=\inf\{|x-y|\,| \,x\in A, \ y\in B\}$
 denotes the Euclidean distance between sets.}:
\begin{enumerate}[(i)]
    \item for every $k\in\mathbb{N}$ the set $\omega^k$ is open and connected;
    \item $\operatorname{diam}\omega^k<\frac12$;
    \item there exists $\varrho>0$, uniform in $\omega$, such that
    \begin{equation}
    \min_{\substack{j\in\mathbb{N},\\ j\neq k}}\operatorname{dist}(\omega^k,\omega^j)>\varrho \operatorname{diam}\omega^k
    \end{equation}
     for all $k\in \mathbb{N}$;
    \item for all $k\in \mathbb{N}$ the sets $(\operatorname{diam} \omega^k)^{-1}\omega^k$ are minimally smooth as per Definition~\ref{def:minimally_smooth} with uniform constants independent of $\omega$.
\end{enumerate}
\end{assumption}

Our geometric assumptions warrant a few remarks.

\begin{remark}
\label{remark geometric assumptions}
\begin{enumerate}[(a)]
    \item In plain English, our assumptions mean that the inclusions are not allowed to be too big, nor to get too close to each other. Item (iii) tells us that how close they are allowed to get to each other depends on the size of the inclusions themselves: smaller inclusions can get closer to one another.
    \item It follows from the discussion after \cite[Assumption~2.6]{CCV2} that (i)--(iv) above are enough to ensure the following:
    one can construct a family $\{\mathcal{B}_\omega^k\}_{k\in\mathbb{N}}$ of mutually non-overlapping bounded open sets in $\R^d$ --- referred to as \emph{extension domains}, cf.~Theorem~\ref{theorem extension} --- such that for every $k\in\mathbb{N}$ we have
    \begin{itemize}
        \item $\omega \cap \mathcal{B}_\omega^k=\omega^k$ and
        \item $(\operatorname{diam} \omega^k)^{-1}(\mathcal{B}_\omega^k\setminus\omega^k)$ is minimally smooth with uniform constants independent of $\omega$ (possibly different from those in Assumption~\ref{assumption 1}(iv)).
    \end{itemize}
    This will play a crucial role in ensuring good extension properties of Sobolev functions defined $\R^d\setminus \omega$ in the sense of Theorem~\ref{theorem extension}.
    \item We should like to emphasise that very small inclusions do not bring about any issues in our analysis. What is troublesome are inclusions that become arbitrarily large. 
\end{enumerate}
\end{remark}

\

We define $\mathcal{F}$ to be the $\sigma$-algebra on $\Omega$ generated by the maps $\pi_q:\Omega \to \{0,1\}$, $q\in \mathbb{Q}^d$, defined in accordance with
\[
\pi_q(\omega):=\chi_\omega(q)\,.
\]

Namely, $\mathcal{F}$ is the smallest $\sigma$-algebra for which the maps $\{\pi_q\}_{q\in \mathbb{Q}}$ are measurable.
Since $\Omega$ is, by definition, invariant under translations, the family of maps $T_y:\Omega \to \Omega$ indexed by $y\in\mathbb{R}^d$ and given by
\begin{equation}
\label{Ty}
    \omega \mapsto T_y\omega=\{z-y\,|\, z\in \omega\}
\end{equation}
is well defined. Clearly, the latter induce a natural group action of $\R^d$ on $\Omega$, and hence on $\mathcal{F}$.
Furthermore, it is easily seen that
\begin{itemize}
    \item $T_{y_1}\circ T_{y_2}=T_{y_1+y_2}$ for all $y_1,y_2\in \mathbb{R}^d$ (here $\circ$ stands for composition) and
    \item the map $(y,\omega)\mapsto T_y\omega$ is measurable with respect to the standard $\sigma$-algebra on $\R^d\times \Omega$ induced by the Borel $\sigma$-algebra on $\R^d$ and $\mathcal{F}$.
\end{itemize}
Observe that $\mathcal{F}$ is \emph{countably} generated; this ensures that $L^p(\Omega)$ is separable for $1\le p<\infty$.

\

We equip $(\Omega,\mathcal{F})$ with a probability measure $\mathbb{P}$ satisfying the following. 

\begin{assumption}
\label{assumption 2}
\begin{enumerate}[(i)]
    \item $\mathbb{P}$ is invariant under the action of $\{T_y\}_{y\in \mathbb{R}^d}$:
\[
\mathbb{P}(T_yF)=\mathbb{P}(F) \quad \forall F\in \mathcal{F}, \qquad T_yF:=\bigcup_{\omega\in F}T_y\omega.
\]
    \item The action of $\{T_y\}_{y\in \mathbb{R}^d}$ on $(\Omega, \mathcal{F}, \mathbb{P})$ is ergodic, i.e., 
    \[
    \mathbb{P}((T_yF \cup F)\setminus(T_yF \cap F))=0 \quad \forall y\in \mathbb{R}^d \quad \Rightarrow \quad \mathbb{P}(F)\in\{0,1\}\,.
    \]
\end{enumerate}
\end{assumption}

Throughout the paper, we adopt the standard notation $\mathbb{E}[f]:=\int_\Omega f(\omega)\,\dr \mathbb{P}(\omega)$ for $f:\Omega \to \mathbb{R}$.

\

The geometric assumptions set out above build up to ensuring the following key property, on which we will rely in an essential fashion in the proof of our main results.

\begin{theorem}[Extension Theorem]
\label{theorem extension}
    Suppose Assumptions~\ref{assumption 1} and~\ref{assumption 2} hold. Then for every $p\ge 1$ there exists a linear bounded extension operator $E_k:W^{1,p}(\mathcal{B}_\omega^k\setminus \omega^k)\to W^{1,p}(\mathcal{B}_\omega^k)$ such that for every $v\in W^{1,p}(\mathcal{B}_\omega^k\setminus \omega^k)$ the extension $v^\ext:=E_k u$ satisfies 
    \begin{enumerate}[(i)]
        \item $(l+x_i)^\ext=l+x_i$ for every $l\in\mathbb{R}$, $x\in\R^d$\,,
        \item $\|v^\ext\|_{L^{p}(\mathcal{B}_\omega^k)}\le c\, \| v\|_{L^{p}\,(\mathcal{B}_\omega^k\setminus \omega^k)}$\,,
        \item $\|\nabla v^\ext\|_{L^p(\mathcal{B}_\omega^k)}\le c\, \|\nabla v\|_{L^p(\mathcal{B}_\omega^k\setminus \omega^k)}$\,,
    \end{enumerate}
    where the constant $c$ depends on $p$ and the parameters of minimal smoothness (entering Assumption~\ref{assumption 2}), but is independent of $\omega$ and $k$. 
\end{theorem}

\begin{remark}
Let us point out that in the current paper we will only extend functions lying in the Sobolev space $W^{1,2}$, namely, we will always be in the case $p=2$. Hence, for all practical purposes, we can suppress the dependence of $c$ from $p$ in Theorem~\ref{theorem extension}, and treat the constant as if it were uniform with respect to $p$.
\end{remark}

Of course, Theorem~\ref{theorem extension} implies functions defined in $W^{1,2}(\R^d\setminus \omega)$ can be extended to functions defined on the whole of $\R^d$. Indeed, one can do this inclusion by inclusion, because extension domains can be constructed in such a way that they do not overlap --- see Remark~\ref{remark geometric assumptions}. With slight abuse of notation, we will be denoting by $v^\ext$ the extension of a function $v$ from any collection of extension domains where it is defined into the corresponding collection of inclusions.

\subsection{Mathematical model}

Put $\M_\omega:=\mathbb{R}^d\setminus \overline{\omega}$. We refer to $\M_\omega$ as the \emph{matrix}, in that it models the ``stiff'' material supporting the random perforation described by $\omega$. Further on, for notational simplicity, we will often tacitly assume the dependence of $\M_\omega$ on $\omega$ and simply write $\M$, instead. 

Given $\omega\in \Omega$, let us consider matrix-valued coefficients $a_\omega:\mathbb{R}^d \to \mathbb{R}^{d^2}$ which are symmetric and satisfy
\begin{equation}
\label{coefficients 1}
   a_\omega =\chi_{\M_\omega}\, a_\omega, 
\end{equation}
\begin{equation}
\label{coefficients 2}
a_-|\xi|^2\le \xi \cdot a_\omega(x)\xi \quad \textrm{and} \quad 
|a_\omega(x) \xi| \le a_+|\xi| \qquad \text{for all}\ x\in \M_\omega,\ \xi \in \R^d. 
\end{equation}

Here 
$a_-,a_+$ are strictly positive constants. 
In plain English, $a_\omega$ is a random matrix coefficient field that vanishes inside the perforation $\omega$ and is elliptic
in the matrix $\M$. We will identify $a_-, a_+$ with the functions that take constant value on $\M$ and vanish on the complement of $\M$ and often write $a$ for $a_\omega$.

\

Given $\overline{f}:\Omega \to \mathbb{R}$, we denote by $f(y,\omega):=\overline{f}(T_y\omega)$ its \emph{stationary extension} (or \emph{realisation}). If $\overline{f}\in L^p(\Omega)$, the latter being the usual space of $p$-integrable functions on $(\Omega, \mathcal{F}, \mathbb{P})$, then $f\in L^p_\loc(\R^d;L^p(\Omega))$ \cite[Chapter~7]{ZKO}.

Recall that a vector field $p\in L^2_\loc(\R^d)$ is said to be
\begin{itemize}
    \item \emph{potential} if there exists $\varphi\in H^1_\loc(\R^d)$ such that $p=\nabla\varphi$;
    \item \emph{solenoidal} if $\int_{\R^d}p\cdot \nabla \varphi=0$ for all $\varphi\in C_0^\infty(\R^d)$.
\end{itemize}
One defines a vector field $\overline{p}\in L^2(\Omega)$ to be potential (resp.~solenoidal) if its stationary extension $y\mapsto f(y,\omega)\in L^2_\loc(\R^d)$ is potential (resp.~solenoidal) almost surely.

Put 
\begin{equation}
\label{V pot}
\mathcal{V}^2_\mathrm{pot}(\Omega):=  \{\overline{p}\in L^2(\Omega)\ |\ \overline{p} \ \text{is potential}, \ \E[\overline{p}]=0\}\,.
\end{equation}
Let $\overline{\psi}\in\mathcal{V}^2_\mathrm{pot}(\Omega)$ be the solution of the problem
\begin{equation}
    \label{corrector equation probability space}
    \E[ a_\omega\, (\psi+\mathbf{e})\cdot \nabla f]=0 \qquad \forall \bar f\in C^\infty(\Omega)\,.
\end{equation}
The existence of $\overline{\psi}\in \mathcal{V}^2_\mathrm{pot}(\Omega)$ is guaranteed
by \cite[\S~8.1]{ZKO} combined with the extension result \cite[Lemma~D.3]{CCV2}.

\begin{definition}
\label{definition corrector}
    We define the \emph{homogenisation corrector} (or \emph{corrector} for short) in the direction $\mathbf{e}\in \mathbb{R}^d$ to be the random variable $\phi=\phi_\omega\in H^1_\mathrm{loc}(\R^d)$ whose gradient is the stationary extension of the solution $\overline{\psi}$ to \eqref{corrector equation probability space}:
    \[
    \nabla\phi(x)=\psi(x,\omega)\,.
    \]
\end{definition}
The corrector $\phi$ is defined uniquely up to an additive constant. Observe that while $\nabla \phi$ is stationary by definition, $\phi$ itself is not, in general, guaranteed to be stationary. In physical space, the corrector satisfies the equation
\begin{equation}
    \label{corrector equation}
    -\nabla \cdot a(\nabla \phi+\mathbf{e})=0\,,
\end{equation}
understood in a weak sense with smooth compactly supported test functions. See also Lemma~\ref{prop:omega_T estimate} for additional details on integrability properties of $\phi$.
We will denote by $\phi^\ext$ the extension of $\phi|_\M$ into $\M$.

\begin{definition}
\label{definition massive corrector}
    We define the \emph{massive homogenisation corrector} (or \emph{massive corrector} for short) on the scale $\sqrt{T}\ge 1$ in the direction $\mathbf{e}\in \mathbb{R}^d$ to be the stationary extension $\phi_T$ of the unique solution $\overline{\phi}_T\in W^{1,2}(\Omega)$ to the equation
  \begin{equation}
    \label{massive corrector equation probability space}
    \E[\chi_\M\frac1T\phi_T f + a_\omega\, (\nabla \phi_T+\mathbf{e})\cdot \nabla f]=0 \qquad \forall f\in C^\infty(\Omega)\,.
\end{equation}
\end{definition}

Realisations of the massive corrector are $H^1_\mathrm{loc}$ functions which satisfy almost surely the equation
\begin{equation}
    \label{massive corrector}
    \chi_\M\frac1T\phi_T-\nabla \cdot a(\nabla \phi_T+\mathbf{e})=0\,,
\end{equation}
understood in a weak sense with smooth compactly supported test functions. 
Additional integrability properties of $\phi_T$ in physical space will be discussed in Lemma~\ref{lemma corrector bounded on balls}. We will denote by $\phi_T^\ext$ the extension of $\phi_T|_\M$ into $\M$.

\

As customary in the subject, the flux $q$ and the massive flux $q_T$ are (random) vector fields are defined in terms of the corrector as
\begin{equation}
\label{definition q}
q:=a(\nabla \phi+\mathbf{e})
\end{equation}
and
\begin{equation}
\label{definition qT}
q_T:=a(\nabla \phi_T+\mathbf{e})\,,
\end{equation}
respectively. One can associate to the flux \eqref{definition q} a tensor field called flux corrector.

\begin{definition}[Flux corrector]
\label{definition flux corrector}
    We define the \emph{flux corrector} to be the unique, modulo an additive constant, tensor field $\sigma\in H^1_\mathrm{loc}(\R^d;L^2(\Omega))$ whose components $\sigma_{ijk}$, $i,j,k\in\{1,\dots,d\}$, are the almost sure solutions of the equation
    \begin{equation}
        \label{flux corrector}
        -\Delta\sigma_{ijk}=\partial_j(q_i)_k-\partial_k (q_i)_j\,,
    \end{equation}
understood in a weak sense with compactly supported test functions. Here $(q_i)_j$ is the $j$-th component of $q_i$, the flux \eqref{definition q} associated with $\mathbf{e}=\hat{e}_i$, where $\hat{e}_i\in \R^d$ is the unit vector along the $i$-th coordinate axis.
\end{definition}

By definition, the flux corrector is skew-symmetric in the last pair of indices, $\sigma_{ijk}=-\sigma_{ikj}$ and it satisfies 
\begin{equation}
\label{eq:sigma_vs_q}
\partial_k\sigma_{ijk}=(q_i)_j\,.
\end{equation} 
As for the corrector, one has that $\nabla \sigma$ is stationary, whereas $\sigma$ itself is not guaranteed to be such. 
Moreover, observe that $\sigma$ is defined through the Poisson equation in the whole space. Notably, the solution $\sigma$ does not directly depend on the specific shape of the inclusions but incorporates their influence solely through the flux $q$. This observation significantly simplifies the analysis by decoupling geometric details from the governing equations. Consequently, the entire theoretical framework developed in~\cite{mjm} can be applied, provided that $q$ satisfies the corresponding bounds required for consistency and regularity.

\

Next, we introduce two quantities that allow one to quantitatively describe the large-scale behaviour of the corrector in an averaged, yet sufficiently precise, sense.

The first is the \emph{excess}, which measures the extent to which $u$ deviates from the space of $a$-affines functions at the level of gradients:
\begin{equation}
\label{definition Exc}
\mathrm{Exc}(\nabla u;B_r):=\inf_{\xi\in\R^d}\fint_{B_{r}} \chi_\M |\nabla u - (\xi + \nabla \phi_\xi)|^2\,.
\end{equation}

The second --- in a sense, the cornerstone of our whole technical analysis --- is $r_*$. Given a constant $\mathcal{C}>0$, we denote by $r_*$ the random variable
\begin{equation}
\label{definition r_*}
r_*=r_*(a_\omega;\mathcal{C}):=\inf\left\{r\ge 1 \ : \ \frac{1}{R^2}\fint_{B_R}\left|(\phi^\ext,\sigma)-\fint (\phi^\ext,\sigma)\right|^2\le \frac1{\mathcal{C}} \quad \forall R\ge r_*  \right\}\,.
\end{equation}
If the RHS of \eqref{definition r_*} is empty, then it is understood that $r_*=+\infty$.

\

Finally, we will introduce an additional assumption on our probability space and present a few meaningful examples satisfying all our assumptions.

\

Given a random variable $F(a)$, one can measure its sensitivity to local perturbations of $a$ by means of its \emph{oscillation}, defined in accordance with
\begin{multline}
\label{oscillation}
\left|\partial^\osc_{Q} F(a) \right|:=\\\sup\left\{F(a_1)-F(a_2)\ \vline\ a_1=a_2=a \quad \text{in}\quad (\R^d\setminus Q)\cup \bigcup_{\omega^k(a)\cap \partial Q \neq \emptyset} \omega^k(a)  \right\}\,.
\end{multline}
That is, $\partial^\osc_{Q}F(a)$ measures how much $F(a)$ varies when one modifies $a$ inside $Q$, but without touching those inclusions that intersect the boundary of $Q$. This is merely conventional for the sake of definiteness: and one could drop the latter requirement without affecting the analysis.

\begin{assumption}[Spectral Gap]
\label{assumption spectral gap}
    There exists constants $\rho,M>0$ such that the probability measure $\mathbb{P}$ satisfies the \emph{Spectral Gap Inequality} with constants $\rho,M$: for every random variable $F=F(a)$ we have
\begin{equation}
\label{spectral gap}
			\operatorname{Var}_L[F]\le \frac{1}{\rho} \,\E_L\left[\int_{\R^d}\left|\partial^\osc_{B_M(x)} F(a) \right|^2\, \dr x\right]\qquad \forall L \geq 1\,,
\end{equation}
where $\mathrm{Var}_L$ and $\E_L$ are variance and expectation taken with respect to the ensemble scaled by $L$, for rescaled coefficient fields $a(L\,\cdot\,)$. 
\end{assumption}

\begin{assumption}[Multi-scale Spectral Gap]
\label{assumption multiscale spectral gap}  
There exists constants $c,\beta,M>0$ such that the probability measure $\mathbb{P}$ satisfies the \emph{Multi-scale Spectral Gap Inequality} with constants $c,\beta,M$: for every random variable $F=F(a)$  we have 

	\begin{equation} 
    \label{multiscale spectral gap}
	\mathrm{Var}_L [F] \leq \E_L \left[ \int_0^{+\infty} \pi (l)(l+1)^{-d} \int_{\R^d}\left|\partial^\osc_{B_{(l+1)M}(x)} F(a) \right|^2\,\dr x\,\dr l	\right] \qquad \forall L\ge 1\,,
	\end{equation}
    where $\pi(l):=c\exp(\frac{l^\beta}{c})$. 
\end{assumption}

\begin{remark} 
\begin{enumerate}[(a)]
    \item It is possible to impose the Multi-scale Spectral Gap Inequality \eqref{multiscale spectral gap} for some integrable $\pi:\R_*\to \R_+$, more generally, obtaining different estimate for the exponential moment of the random variable $r_*$ (we refer the reader to~\cite{mjm} for details). However, in the examples we shall present we will only consider $\pi$ in the form above. Note that the ``ordinary'' Spectral Gap Inequality \eqref{spectral gap} is satisfied when \eqref{multiscale spectral gap} holds for a $\pi$ with compact support. 
    \item We require the Spectral Gap and Multi-scale Spectral Gap Inequalities to hold for every $L\geq  1$ rather than just for $L=1$. One can show that in all examples of interest such condition is satisfied. This is connected with the fact that mixing properties improve under scaling with $L \geq 1$, see also \cite[Remark 2.2]{DG2}.
    \item Instead of imposing the (Multi-scale) Spectral Gap Inequality, one could impose a stronger (Multi-scale) Logarithmic Sobolev Inequality, see, e.g.,~\cite{DG1}. This would again imply a stronger estimate for the exponential moments of the random variable $r_*$. In Theorem~\ref{theorem 2} one could have $d/2$ replaced with $d$, provided the weight $\pi$ possesses exponential decay as imposed above. 	
\end{enumerate}
\end{remark}

\subsubsection{Some examples}

We will now discuss some prototypical examples to demonstrate that our set of assumptions is meaningful and covers a non-empty class of scenarios. These examples illustrate how the (Multi-Scale) Spectral Gap Inequality applies in various random settings.

\subsubsection*{Example 1} 
Let us consider a sequence of i.i.d.~random variables $(V_j)_{j \in \Z^d}$ taking values in $[0,1/3]$. Construct a random perforation by placing a hole $B_{V_j}(z)$ at each $z \in \mathbb{Z}^d$. In this case, the Spectral Gap Inequality~\eqref{spectral gap} is satisfied with $M = 2/3 + \varepsilon$, for all $\varepsilon > 0$, as shown in~\cite[Proposition 2.3]{DG2}. 

\subsubsection*{Example 2} 
Consider a Poisson point process $\mathcal{P}$ with intensity $\mu > 0$. Denoting the points of the process by $\{p_k\}_k$, construct a random perforation by placing a hole of radius 
$r_k := \min\left\{R, \frac{1}{3}\inf_{j\, :\, j \neq k} |p_k - p_j|\right\}$
at each $p_k \in \R^d$. Here, $ R > 0 $ acts as an upper bound on the hole radii. In view of~\cite[Proposition 2.3]{DG2}, the Spectral Gap Inequality~\eqref{spectral gap} is satisfied with $ M = 2R + \varepsilon $, for every $ \varepsilon > 0 $. 

\subsubsection*{Example 3} 
Consider a random parking measure as a point process. This process ensures that there exists $ R > 0 $ such that all points are at least distance $ R $ from one another. Furthermore, it is saturated: no additional point can be added to the configuration while maintaining this minimum distance constraint. Construct a perforation by placing a hole of radius $ R/2 $ at every point of the process. Then, the Multi-Scale Spectral Gap Inequality~\eqref{multiscale spectral gap} is satisfied with weights given by 
$\pi(l) = C \exp(-l / C),$
where $ C > 0 $ depends on structural properties of the process. 

\vspace{\baselineskip}

Further examples can be found in~\cite{DG2}, including random point processes with randomly shaped holes and random radii. The corresponding formulae for the weight function $ \pi(l) $ are also discussed there. Finally, we emphasize that while these (Multi-Scale) Spectral Gap Inequalities hold under specific geometric or probabilistic conditions, they are not guaranteed if one only assumes finite correlation properties. Each scenario must be carefully examined based on its unique structure and assumptions.

\section{Main results}
\label{Main result}

The main result of our paper is the following bound on the exponential moments of the regularity radius $r_*$ for perforated domains. 

\begin{theorem}
\label{theorem 2}
Suppose that the probability measure $P$ satisfies the Multi-scale Spectral Gap Inequality~\eqref{multiscale spectral gap}.
Then there exists a positive constant $C=C(d,a_\pm,\beta)$ such that
\begin{equation} \label{theorem 2 equation 1} 
			\E\left[ e^{r_*^{\min\{d/2,\beta\}}/C} \right]\le 2.
		\end{equation}
		In the case when spectral gap is satisfied, we can take $\beta=\infty$ in \eqref{theorem 2 equation 1}. 
	\end{theorem}

From the bounds of the exponential moments, it is well known that one can obtain bounds on the growth of the corrector and, in turn, estimate the homogenization error. This is the content of the upcoming statements. Note that Theorem~\ref{theorem 2} supplies an essentially optimal bound. The estimates below, however, are suboptimal, and can be improved by using more sophisticated approaches available in the literature, e.g., that from~\cite{gloria2012quantitative}.

\begin{theorem}
\label{bella kniely theorem 1.13}
   Suppose that Assumption~\ref{assumption multiscale spectral gap} is satisfied. Let $\epsilon\in(0,1)$ be the constant given by the Hole Filling Lemma~\ref{lemma hole filling}.   Then there exists a random field $\mathfrak{C}(x)$ with stretched exponential moments
   \begin{equation}
       \label{bella kniely theorem 1.13 equation 1}
    \E\left[\exp\left(\frac1C \mathfrak{C}^{\gamma}\right)\right]\le 2
   \end{equation}
   for a sufficiently large constant $C$ and $$\gamma:=\min\left\{\frac12,\frac{\min\{\frac{d}2,\beta\}}{d(1-\epsilon)}\right\},$$ such that the corrector satisfies the inequality
   
   \begin{equation}
       \label{bella kniely theorem 1.13 equation 2}
\left(\fint_{B_1(x)}|\phi|^2 \right)^\frac12\lesssim \left|\fint_{B_1}\phi \right|+\mathfrak{C}(x)\pi(|x|)\,,
   \end{equation}
   with 
   \begin{equation}
        \label{bella kniely theorem 1.13 equation 3}
        \pi(r):=
        \begin{cases}
           r^{\frac{\epsilon d}{2}(\frac{2}{\epsilon d}-1)} & \text{if }1-\frac{2}{\epsilon d}<0\\
            \ln(2+r) & \text{if }1-\frac{2}{\epsilon d}=0\\
             1 & \text{if }1-\frac{2}{\epsilon d}>0\\
        \end{cases}.
   \end{equation}
   \end{theorem}
   

\begin{proof}[Proof of Theorem~\ref{bella kniely theorem 1.13}]
    The proof retraces the steps of \cite[Theorem~1.13]{bella_kniely}, the key ingredient being the sensitivity estimate obtained in Proposition~\ref{bella kniely proposition 1.12}\,.
\end{proof}

\bbb

Let us know state a direct consequence of the corrector bounds stated above:

\begin{corollary}
\label{corollary 1.14}

Let $u_\varepsilon$ and $u_\mathrm{hom}$ be the solutions of
\[
-\nabla \cdot a_{T_{\cdot/\epsilon}\omega}\nabla u_\epsilon=\nabla \cdot (\chi_{\M_{\cdot/\epsilon}\omega}\,g), \qquad 
-\nabla \cdot a_\mathrm{hom}\nabla u_\hom=\nabla \cdot g, \qquad a_\hom\mathbf{e}:=\E[a(\nabla\phi+\mathbf{e})]\,,
\]
where $g\in H^1(\R^d)$. Let
\[
z_\varepsilon:=u_\varepsilon-
\left.\left(u{_\mathrm{hom},\varepsilon}+\varepsilon(\phi_i)_{T_{\cdot/\epsilon}\omega}\partial_iu_{\mathrm{hom},\varepsilon}\right)\right|_{\M_{\cdot/\epsilon}\omega}, \qquad u_{\mathrm{hom},\varepsilon}(x):=\fint_{B_\varepsilon(x)} u_{\mathrm{hom}}\,.
\]
Then we have
\begin{equation}
    \|\nabla z_\varepsilon\|_{L^2} \le \varepsilon\,\|\nabla g\|_{L^2}
    +
    \mathfrak{C}_{\varepsilon,g} \,\varepsilon\,\pi(\varepsilon^{-1})\|\pi(|x|)\nabla g\|_{L^2}\,,
\end{equation}
where $\mathfrak{C}_{\varepsilon,g}$ is a random variable satisfying
\begin{equation}
    \E\left[\exp\left(\frac1C\mathfrak{C}_{\varepsilon,g}^{\gamma}\right)\right]<2
\end{equation}
for a sufficiently large constant $C$, and $\gamma$ is defined as in Theorem~\ref{bella kniely theorem 1.13}.
\end{corollary}

   The proof to the above corollary follows a classical argument in the subject based on a two-scale expansion of the solution, see, e.g., \cite[Corollary~1.14]{bella_kniely}, hence we omit the details for the sake of brevity. 

\section{Proof of Theorem~\ref{theorem 2}}
\label{Proof of main theorem}

This section contains a collection of results that build up to the proof of our main statement, Theorem~\ref{theorem 2}. Our overall strategy follows closely pioneering ideas from \cite{mjm}. Indeed, the statements below are counterparts of statements from \cite{mjm} in the perforated, rather than uniformly elliptic, setting.

\

The first statement establishes regularity properties of $a$-harmonic functions depending on the quantification of the (sub)linear growth of the corrector expressed by the random variable $\delta$ defined in accordance with \eqref{proposition 1 equation 3}. Note that qualitative analysis tells us that $\delta \to 0$ almost surely as the scale at which we measure the linear growth gets larger and larger.

\begin{proposition}
\label{proposition 1}
Suppose that $u$ is an $a$-harmonic function in a Lipschitz domain $\tilde B_R$ such that $B_{R}\subset \tilde B_R\subset B_{R+1}$ and $\partial \tilde B_R \cap \omega=\emptyset$, with $R>4$. Namely,
\begin{equation}
\label{proposition 1 equation 1}
-\nabla \cdot a \nabla u=0 \qquad \text{in} \ \tilde B_{R}\,.
\end{equation}
Then for all $1 < r\le R/4$ and $1 < \rho \le R/8$ there exists $\xi\in \R^d$ such that
\begin{multline}
\label{proposition 1 equation 2}
	\fint_{B_{r}} \chi_\M |\nabla u - (\xi + \nabla \phi_\xi)|^2 
\\\leq C(d,a_\pm)  \left(  (1+\delta) \frac{r^2}{R^2} + \frac{R^d}{r^d}\left( \delta\left(\frac{R}{\rho} \right)^{d+2}     +\frac{\rho}{R}\right)  \right) \fint_{B_{R}}\chi_\M|\nabla  u|^2\,,
\end{multline}
where
\begin{equation}
   \label{proposition 1 equation 3}
	\delta:=\max\left\{ \frac{1}{R^2} \fint_{B_R}  |(\phi^\ext,\sigma) - \fint_{B_R} (\phi^\ext,\sigma)|^2, \, \frac{1}{r^2} \fint_{B_{2r}}  |\phi^\ext - \fint_{B_{2r}} \phi^\ext|^2, 
\right\}. 
\end{equation}
Furthermore, we have
\begin{equation}
\label{proposition 1 equation 4}
\frac1{C(d,a_\pm)}\left(1-C(d) \delta^{1/4}\right)^3|\xi|^2 \le \fint_{B_r} \chi_\M |\xi+\nabla\phi_\xi|^2\le C(d,a_\pm) (1+\delta) |\xi|^2 \quad \forall\xi \in \R^d.
\end{equation}
\end{proposition}

\begin{proof}
Let us start by proving~\eqref{proposition 1 equation 2}.
 Without loss of generality, we assume that $\fint_{B_R} (\phi^\ext,\sigma) = 0$.
 We begin with  Caccioppoli's estimate which will be utilised later. Testing equation \eqref{proposition 1 equation 1} with $\eta^2 u$, where $\eta$ is a cut-off function between $B_{\hat R-\hat \rho}$ and $B_{\hat R}$, for some  $0<\hat\rho<\hat R\leq R$ is some positive number, and using the identity 
\begin{equation}
	a\nabla(\eta u) \cdot \nabla(\eta u) = a\nabla u \cdot \nabla(\eta^2 u) + ua  \nabla \eta \cdot \nabla(\eta u) - u a \nabla (\eta u) \cdot \nabla \eta + u^2 a\nabla\eta \cdot \nabla\eta,
\end{equation} 
we get the bound
\begin{equation}
	\int \chi_\M |\nabla(\eta u)|^2 \lesssim  \int \chi_\M  |u \nabla \eta|^2
\end{equation}
and, consequently,
\begin{equation}\label{1.6}
	\int_{B_{\hat R-\hat\rho}} \chi_\M |\nabla u|^2 \lesssim \frac{1}{\hat\rho^2} \int_{B_R \setminus B_{\hat R-\hat\rho}} \chi_\M u^2.
\end{equation}

Using Fubini we see that for $T>2$ and $f\in L^1(0,T)$, $f\geq 0$ 
\begin{equation}\label{eqPeter}
\int_1^T \int_{t-1}^t f(s)\ds \dt \leq \int_0^T f(t)\dt.
\end{equation}

Let $u^\ext$ be the extension of $u|_{\M\cap B_{R}}$ into $B_{R-1}$. Applying~\eqref{eqPeter} to $f(t):=\int_{\partial B_{t}}  |\nabla u^\ext|^2$ we have
\begin{equation}
	\int_{R/2}^{R-1} \Big(\int_{\partial B_{t}}  |\nabla u^\ext|^2 + \int_{B_{t}\setminus B_{t-1}} |\nabla  u^\ext |^2\Big) dt \leq 2 \int_{B_{R-1} } |\nabla u^\ext|^2.
\end{equation} 
From this we immediately conclude  that there exist a radius $R'\in (\frac{1}{2}R,R-1)$ such that \footnote{Here one takes the precise representative of $u^\ext$, see e.g.  \cite{EG}, in particular, the tangential gradient $\nabla^{tan} u^\ext \in L^2(\partial B_{R'})$ is well defined for every $R'$.}
\begin{equation}\label{1.7}
	\int_{\partial B_{R'}}  |\nabla u^\ext|^2 \lesssim \frac{1}{R} \int_{B_{R-1} } |\nabla u^\ext|^2 \lesssim \frac{1}{R} \int_{B_{R}} \chi_\M |\nabla  u|^2,
\end{equation}
and
\begin{equation}\label{1.8a}
	\int_{B_{R'}\setminus B_{R'-1}} |\nabla  u^\ext |^2 \lesssim \frac{1}{R} \int_{B_{R-1} } |\nabla u^\ext|^2 \lesssim \frac{1}{R} \int_{B_{R}} \chi_\M |\nabla  u|^2.
\end{equation}

Let $u_\hom$ be the $a_\hom$-harmonic extension of $u^\ext|_{B_{R-1}\setminus B_{R'}}$ into $B_{R'}$:
\begin{equation}\label{1.8}
	-\nabla\cdot a_\hom\nabla u_\hom = 0 \mbox{ in } B_{R'},\quad u_\hom = u^\ext \mbox{ on }\partial B_{R'}.
\end{equation}
We  approach  the bound \eqref{proposition 1 equation 2} by considering first the difference between $u$ and $u_\hom + \eta \partial_i u_\hom  \phi_i$ in the ball $B_{R'}$, where $\eta $ is a cut-off function between $B_{R'-2\rho}$ and $B_{R'-\rho}$, where we additionally assume that $1<\rho\leq R'/4$. However, due to subtleties of controlling the extension into the perforations, we need to define it in a specific way. 
\begin{equation}
	v:=\left\{ 
	\begin{array}{ll}
		u^\ext - u_\hom, & \mbox{ in } \bigcup_{\B_\omega^k \cap \partial B_{R'} \neq \emptyset} \omega^k,
		\\
		u - (u_\hom + \eta \partial_i u_\hom  \phi_i), &  \mbox{ in } \M \cap B_{R'},
		\\
		u^\ext - (u^\ext_\hom + (\eta \partial_i u_\hom  \phi_i)^\ext), & \mbox{ in } \bigcup_{\B_\omega^k \subset  B_{R'} } \omega^k.
	\end{array}
	\right.
\end{equation}
Note that $v\in H^1_0(B_{R'})$. In the above the superscript $\cdot^\ext$ stands for the extension of the restriction of the function to the set $\M$ into the perforations $\omega^k$.
Although the extension into  the perforations does not appear in the next few calculations, it will play an important role at a later stage. 

The following calculation is to be understood in the weak sense (i.e. it holds for any test function from $H^1(\R^d)$):
\begin{equation}
	-\nabla\cdot a \nabla v = \nabla \cdot (1-\eta)a \nabla  u_\hom + \nabla (\eta \partial_i u_\hom) \cdot a(\nabla \phi_i + e_i) + \nabla \cdot \phi_i a \nabla (\eta \partial_i u_\hom).
\end{equation}
In the above we used that $-\nabla\cdot a (\nabla \phi_i + e_i) =0$.
Utilizing \eqref{1.8} we have
\begin{equation}
	\nabla (\eta \partial_i u_\hom) \cdot a_\hom  e_i = - \nabla \cdot ((1-\eta) a_\hom \nabla  u_\hom) 
\end{equation}
(here we have actual $u_\hom$ everywhere rather than its harmonic extension).
Combining the last two identities yields
\begin{equation}
	-\nabla\cdot a \nabla v = \nabla \cdot ((1-\eta)(a - a_\hom) \nabla  u_\hom + \nabla (\eta \partial_i u_\hom) \cdot (a(\nabla \phi_i + e_i) - a_\hom e_i) + \nabla \cdot \phi_i a \nabla (\eta \partial_i u_\hom).
\end{equation}
Next, recalling~\eqref{eq:sigma_vs_q}, we make a substitution
\begin{equation}
	\nabla\cdot\sigma_i = q_i = a(\nabla \phi_i +e_i) - a_\hom e_i,
\end{equation}
and use the skew-symmetry of $\sigma_i$ to get
\begin{equation}
	-\nabla\cdot a \nabla v = \nabla \cdot ((1-\eta)(a - a_\hom) \nabla  u_\hom + (\phi_i a - \sigma_i)\nabla (\eta \partial_i u_\hom)).
\end{equation}
Testing this identity with $v$ and using H\"older's inequality we have
\begin{equation}
	\int_{B_{R'}}  a \nabla v\cdot \nabla v = 	\int_{B_{R'}} ((1-\eta)(a - a_\hom) \nabla u_\hom + (\phi_i a - \sigma_i)\nabla (\eta \partial_i u_\hom))\cdot \nabla v,
\end{equation}
\begin{multline}
	\int_{B_{R'}} \chi_\M |\nabla v|^2 \lesssim  \|(1-\eta) \nabla u_\hom\|_{B_{R'}} \|\nabla v \|_{B_{R'}\setminus B_{R'-2\rho}}+
	\\ 
	\|\chi_\M \phi \nabla(\eta \nabla u_\hom)\|_{B_{R'-\rho}} 	\|\nabla v\|_{B_{R'-\rho}} + \|\sigma \nabla(\eta \nabla u_\hom)\|_{B_{R'-\rho}} 	\|\nabla v\|_{B_{R'-\rho}}.
\end{multline}
By construction of $v$ and  properties of the extension operator we have
\begin{equation}
	\|\nabla v\|_{B_{R'-1}} \lesssim \|\chi_\M \nabla v\|_{B_{R'}}.
\end{equation}
Then, using Young's inequality we can absorb this term by the left hand side:
\begin{multline}
	\int_{B_{R'}} \chi_\M |\nabla v|^2 
	\lesssim  \|(1-\eta) \nabla u_\hom\|_{B_{R'}}^2 +
	\|\chi_\M \phi \nabla(\eta \nabla u_\hom)\|_{B_{R'-\rho}}^2 
	\\
	+ \|\sigma \nabla(\eta \nabla u_\hom)\|_{B_{R'-\rho}}^2 
	+  \|\nabla v \|_{B_{R'}\setminus B_{R'-1}}^2.
\end{multline}
Then $\chi_\M \nabla v = 	\chi_\M( \nabla u - (\nabla u_\hom + \eta \partial_i u_\hom  \nabla \phi_i + \phi_i \nabla (\eta \partial_i u_\hom) ))$ yields
\begin{multline}
	\int_{B_{R'-2\rho}} \chi_\M |\nabla u - \partial_i u_\hom (e_i+ \nabla \phi_i)|^2 
	\lesssim  \int_{B_{R'}\setminus B_{R'-2\rho}} |\nabla u_\hom|^2 +
	\\
	+ \int_{B_{R'-\rho}} |(\chi_\M \phi ,\sigma)|^2 (|\nabla^2 u_\hom|^2 + \rho^{-2}|\nabla u_\hom|^2 ) 
	+  \int_{B_{R'}\setminus B_{R'-1}} |\nabla v |^2.
\end{multline}
For $r\leq \frac{R'}{2} \leq R'-2\rho$ and setting $\xi = \nabla u_\hom(0)$ we have 
\begin{multline}\label{1.21}
	\int_{B_{r}} \chi_\M |\nabla u - (\xi + \nabla \phi_\xi)|^2 
	\\
	\lesssim  	\int_{B_{r}} \chi_\M |I+\nabla \phi|^2|\nabla u_\hom  - \nabla u_\hom(0)|^2 + 
	\int_{B_{R'}\setminus B_{R'-2\rho}} |\nabla u_\hom|^2 
	\\
	+ \int_{B_{R'-\rho}} |(\chi_\M \phi ,\sigma)|^2 (|\nabla^2 u_\hom|^2 + \rho^{-2}|\nabla u_\hom|^2 ) 
	+
	\int_{B_{R'}\setminus B_{R'-1}} |\nabla u^\ext |^2
	\\
	\lesssim r^2 \sup_{B_r} |\nabla^2 u_\hom|^2 	\int_{B_{r}} \chi_\M |id+\nabla \phi|^2 
	\\
	+	 \sup_{B_{R'-\rho}}  (|\nabla^2 u_\hom|^2 + \rho^{-2}|\nabla u_\hom|^2 )  \int_{B_{R'-\rho}} |(\chi_\M \phi ,\sigma)|^2 
	\\
	+	\int_{B_{R'}\setminus B_{R'-2\rho}} |\nabla u_\hom|^2 
	+
	\int_{B_{R'}\setminus B_{R'-1}} |\nabla u^\ext |^2.
\end{multline}

Next we derive bounds for the terms on the right hand side of \eqref{1.21}.

Testing \eqref{1.8} with $u_\hom - u^\ext$ we have
\begin{equation}\label{1.23}
	\|\nabla u_\hom\|_{B_{R'}} \lesssim 	\|\nabla u^\ext\|_{B_{R'}}   \lesssim 	\|\chi_\M\nabla  u\|_{B_{R}}.  
\end{equation}

The following estimates have been shown to hold for an $a_\hom$-harmonic function on $B_{R}$ in \cite{Giaquinta} (here $R$ is generic):
\begin{equation}\label{1.24}
	R^2 \sup_{B_{R/2}} |\nabla^2 u_\hom|^2 + \sup_{B_{R/2}} |\nabla u_\hom|^2 \lesssim \fint_{B_{R}}   |\nabla u_\hom|^2, 
\end{equation}
\begin{equation}\label{1.25}
	\int_{B_{R'}\setminus B_{R'-2\rho}} |\nabla u_\hom|^2  \lesssim \rho \int_{\partial B_{R'}} |\nabla^{tan} u_\hom|^2.
\end{equation}
Combining \eqref{1.23}, \eqref{1.24} and recalling that  $r \leq R'/2$ and $R'\in (R/2,R-1)$ we arrive at
\begin{equation}\label{1.26}
	r^2 \sup_{B_r} |\nabla^2 u_\hom|^2 \lesssim \left(\frac{r}{R'}\right)^2 \fint_{B_{R'}} |\nabla u_\hom|^2 \lesssim  \left(\frac{r}{R}\right)^2 \fint_{B_{R}} \chi_\M |\nabla  u|^2.
\end{equation}
Applying \eqref{1.24} in the form 
\begin{equation*}
	\rho^2 |\nabla^2 u_\hom(x)|^2 + |\nabla u_\hom(x)|^2 \lesssim \fint_{B_{\rho}(x)}   |\nabla u_\hom|^2, \quad x \in B_{R'-\rho},
\end{equation*}
we have for all $x \in B_{R'-\rho}$
\begin{multline}
	|\nabla^2 u_\hom|^2 (x)+ \rho^{-2}|\nabla u_\hom|^2 (x)  \lesssim \frac{1}{\rho^{2}} \fint_{B_{\rho}(x)} |\nabla u_\hom|^2 
	\\ \leq \frac{1}{\rho^{2}} \left(\frac{R'}{\rho}\right)^d\fint_{B_{R'}} |\nabla u_\hom|^2 
	\lesssim  \frac{1}{\rho^{2}} \left(\frac{R}{\rho}\right)^d \fint_{B_{R}} \chi_\M|\nabla u|^2.
\end{multline}
Combining \eqref{1.25} with \eqref{1.7}
\begin{equation}\label{1.28}
	\int_{B_{R'}\setminus B_{R'-2\rho}} |\nabla u_\hom|^2  \lesssim \rho \int_{\partial B_{R'}} |\nabla^{tan} u_\hom|^2 
	\lesssim \frac{\rho}{R}  \int_{B_{R}} \chi_\M |\nabla  u|^2.
\end{equation}
Now applying \eqref{1.8a} and \eqref{1.26}-\eqref{1.28} to the right hand side of \eqref{1.21}
we arrive at
\begin{multline}
	\fint_{B_{r}} \chi_\M |\nabla u - (\xi + \nabla \phi_\xi)|^2 
	\lesssim\frac{r^2}{R^2} 	\fint_{B_{r}} \chi_\M |I+\nabla \phi|^2 \fint_{B_{R}} \chi_\M |\nabla  u|^2 
	\\
	+	 \frac{R^d}{r^d} \left( \left(\frac{R}{\rho}\right)^{d+2} \frac{1}{R^2}\fint_{B_{R}} |(\chi_\M \phi ,\sigma)|^2 + \frac{\rho}{R} \right) \fint_{B_{R}} \chi_\M |\nabla  u|^2.
\end{multline}

Applying Caccioppoli's estimate \eqref{1.6} to the function $x_i + \phi_i - c_i$ (observe that one can choose $c_i=\fint_{B_{2r}\cap \M}\phi_i$), we have
\begin{equation}\label{3.58}
	\fint_{B_r} \chi_\M  |e_i + \nabla \phi_i |^2 \lesssim \frac{1}{r^2} \fint_{B_{2r}} \chi_\M  |x_i + \phi_i - c_i|^2 \lesssim 1+  \frac{1}{r^2} \fint_{B_{2r}}  | \phi_i^\ext - \fint_{B_{2r}} \phi_i^\ext|^2.
\end{equation}

So we arrive at
\begin{multline}
	\fint_{B_{r}} \chi_\M |\nabla u - (\xi + \nabla \phi_\xi)|^2 
	\leq C \left( (1+\delta) \frac{r^2}{R^2} +  \frac{R^d}{r^d}\left( \delta\left(\frac{R}{\rho} \right)^{d+2}     +\frac{\rho}{R}\right)  	\right) \fint_{B_{R}}\chi_\M|\nabla  u|^2.
\end{multline}

Let us prove \eqref{proposition 1 equation 4}. The upper bound is a consequence of  \eqref{3.58}. Thus, it only remains to show the lower bound. Let $\eta $ be a  cut-off function   between $B_{r-1-\tau}$ and $B_{r-1}$ with $0<\frac{\tau}{r}$ sufficiently small. Applying the Extension Theorem to $(\xi+\nabla\phi_\xi)|_\M$ (recall that our extension preserves affine functions) and Jensen's inequality in the form $(\int \eta)^{-1} \int \eta f^2 \geq ((\int \eta)^{-1} \int \eta f)^2$  we have
\begin{multline}\label{3.60}
	\int_{B_r} \chi_\M |\xi+\nabla\phi_\xi|^2 
	\ge 
	C\int_{B_{r-1}} \eta|\xi+\nabla\phi_\xi^\ext|^2
\ge
	C\left( \int  \eta \right) \left|\xi+\left(\int \eta\right)^{-1}\int_{B_{r-1}} \eta \nabla\phi_\xi^\ext \right|^2
	\\
	=
C	\left(\int  \eta\right)\left|\xi-\left(\int   \eta\right)^{-1}\int_{B_{r-1}} \nabla\eta\, (\phi_\xi^\ext-\fint_{B_{r-1}} \phi_\xi^\ext)\right|^2.
\end{multline}
Further, one can easily check that 
\begin{equation}\label{3.61}
	\int  \eta \geq C(d) r^d(1 - C(d) \frac{\tau}{r}).
\end{equation}
Then we have
\begin{multline}\label{3.62}
	\left(\int   \eta\right)^{-1} \left|\int_{B_{r-1}} \nabla\eta \left(\phi_\xi^\ext-\fint_{B_{r-1}} \phi_\xi^\ext\right)\right|\leq C(d) \frac{1}{\tau r^d} \int_{B_{r-1}} \left|   \phi_\xi^\ext-\fint_{B_{r-1}}   \phi_\xi^\ext\right| 
	\\
	\leq  C(d) \frac{1}{\tau } \left( \fint_{B_{r-1}} \left|   \phi_\xi^\ext-\fint_{B_{r-1}}   \phi_\xi^\ext\right|^2\right)^{-1/2}	\leq  C(d) \frac{r}{\tau } \delta^{1/2} |\xi|.
\end{multline}
Choosing $\tau$ such that  $\frac{\tau}{r} = \delta^{1/4}$ and combining \eqref{3.60}--\eqref{3.62} we arrive at \eqref{proposition 1 equation 4}.
\end{proof}
\bbb

The next two statements --- Proposition~\ref{theorem 1} and Proposition~\ref{corollary 3} --- establish large-scale Schauder estimates for solutions of degenerate elliptic PDEs in divergence form, which will later be needed in the proof of Proposition~\ref{proposition 2}. Furthermore, Proposition~\ref{theorem 1} establishes growth properties (sublinearlty) of the corrector on scales bigger than the regularity radius $r_*$ in the sense of \eqref{theorem 1 equation 3}.

\begin{proposition}
\label{theorem 1}
For any $\alpha\in(0,1)$ there exists constants $C_1(d,a_\pm,\alpha)$, $C_2(d,a_\pm,\alpha)$  with the following properties. Let $r_*$ be defined as in \eqref{definition r_*} with $\mathcal{C}=C_1(d,a_\pm,\alpha)$. If 
\begin{equation}
\label{theorem 1 equation 1}
-\nabla \cdot a \nabla u=0 \qquad \text{in} \ \tilde B_{R}
\end{equation}
for some $R\ge r_*$, where $\tilde B_R$ is defined as in Proposition~\ref{proposition 1}, then we have
\begin{equation}
\label{theorem 1 equation 2}
\mathrm{Exc}(\nabla u;B_r)\le C_2(d,a_\pm, \alpha) \left(\frac{r}{R}\right)^\alpha \mathrm{Exc}(\nabla u; B_R) \qquad \forall r\in[r_*,R]
\end{equation}
and
\begin{equation}
\label{theorem 1 equation 3}
\frac1C|\xi|^2 \le \fint_{B_r} \chi_\M |\xi+\nabla\phi_\xi|^2\le C_2(d,a_\pm,\alpha)|\xi|^2 \qquad \forall r\ge r_*, \forall\xi\in \R^d\,.
\end{equation}

Furthermore, fixing $\alpha>0$ (say, $\alpha=\frac12$), we have the mean value property
\begin{equation}
\label{theorem 1 equation 4}
\fint_{B_r}\chi_\M |\nabla u|^2 \le C(d,a_\pm) \fint_{B_R} \chi_\M|\nabla u|^2 \qquad \forall r\in[ r_*,R].
\end{equation}
\end{proposition}

\begin{proof}

Fix $\alpha\in(0,1)$ and let $\delta=\delta(d,a_\pm,\alpha)\in(0,1)$ be a sufficiently small parameter to be chosen later. 

First note that for $\rho >1$ one has
\begin{equation}
	\label{proof proposition 1 equation 3}
	\fint_{B_{\rho}}\chi_\M \left| \phi-\fint_{B_{\rho}\cap \M}\phi\right|^2\le C \fint_{B_{\rho}} \left| \phi^\ext-\fint_{B_{\rho}}\phi^\ext\right|^2\,.
\end{equation}
Indeed, this bound is a consequence of the trivial  estimate
$\fint_{B_{\rho}}\chi_\M \left| \phi-c\right|^2\le C \fint_{B_{\rho}} \left| \phi^\ext-c\right|^2$
after one chooses
$c=\fint_{B_{\rho}}\phi^\ext$ in the right-hand side and optimises for $c$ in the left-hand side. Choosing sufficiently large $\mathcal C$ in \eqref{definition r_*} we guarantee that 
\begin{equation}
	\frac{1}{R'^2} \fint_{B_R'}  \left|(\phi^\ext,\sigma) - \fint_{B_R'} (\phi^\ext,\sigma)\right|^2 \le \delta \qquad \forall R'\ge r_*.
\end{equation}
Then by Proposition \ref{proposition 1} one has
\begin{multline}
\label{proof proposition 1 equation 2}
	\fint_{B_{r}} \chi_\M |\nabla u - (\xi + \nabla \phi_\xi)|^2 
\\
\leq C \left((1+\delta) \frac{r^2}{R'^2} + \frac{R'^d}{r^d}\left( \delta\left(\frac{R'}{\rho} \right)^{d+2}     +\frac{\rho}{R'}\right)  \right) \fint_{B_{R'}}\chi_\M|\nabla  u|^2
\end{multline}
for all $r_*\le r\le R'\le R$, where the constant only depends on $d$ and $a_\pm$.

Upon replacing $u\mapsto u -(\xi\cdot x+\phi_\xi)$ (which is also $a$-harmonic) and optimising w.r.t $\xi$ one obtains from \eqref{proof proposition 1 equation 2} that
\begin{equation}
\label{proof proposition 1 equation 4}
	\mathrm{Exc}(\nabla u;B_r)
\leq C\left((1+\delta) \frac{r^2}{R'^2} + \frac{R'^d}{r^d}\left( \delta\left(\frac{R'}{\rho} \right)^{d+2}     +\frac{\rho}{R'}\right)  \right)\mathrm{Exc}(\nabla u;B_{R'})\,.
\end{equation}
Let $\theta >0$ so that $2C\theta^2 \le \frac13 \theta^{2\alpha}$.
Choose $\delta$ so that $C\delta^{\frac{1}{d+3}}\theta^{-d}\le \frac13 \theta^{2\alpha}$. Denote $R^*(\delta) := \delta^{-\frac{1}{d+3}}$, and assume that $R' \ge R^*(\delta)$. Then set
\begin{equation}
	\rho=  \delta^{\frac{1}{d+3}} R' \ge 1.
\end{equation}
Notice that for such $\rho$ one has
\begin{equation}
 \delta\left(\frac{R'}{\rho} \right)^{d+2}     +\frac{\rho}{R'}=  2 \delta^{\frac{1}{d+3}}.
\end{equation}
Now choose $r$ so that $\theta=\frac{r}{R'}$. Under these assumptions  \eqref{proof proposition 1 equation 4} takes the form
\begin{equation}
\label{proof proposition 1 equation 5}
	\mathrm{Exc}(\nabla u;B_{\theta R'})
\leq \theta^{2\alpha}\, \mathrm{Exc}(\nabla u;B_{R'})\, \qquad \forall R' \ge R^*(\delta), \ \theta R'\ge r_*.
\end{equation}
The estimate \eqref{proof proposition 1 equation 5} can be iterated to obtain
\begin{equation}
\label{proof proposition 1 equation 6}
	\mathrm{Exc}(\nabla u;B_{\theta^n R'})
\leq \theta^{2n\alpha}\, \mathrm{Exc}(\nabla u;B_{R'})\, \qquad \forall n\in\mathbb{N}\ \text{s.t.} \ \theta^{n} R'\ge \max\{R^*(\delta), r_*\}.
\end{equation}
For $\max\{R^*(\delta), r_*\}\le r\le R$, let $ n\in \mathbb{N}\cup \{0\}$ be such that $\theta^{n+1}<\frac{r}{R}\le \theta^n$. Observe that this implies $\theta^n \frac{R}{r}<\frac1\theta$. By combining the latter inequality with the trivial estimate $\fint_{B_r}\lesssim \left(\frac{\theta^n R}{r}\right)^d \fint_{B_{\theta^n R}}$ and using \eqref{proof proposition 1 equation 6} with $R'\mapsto R$, we arrive at
\begin{multline}
\label{proof proposition 1 equation 7}
	\mathrm{Exc}(\nabla u;B_{r})
\leq \theta^{-d}\, \mathrm{Exc}(\nabla u;B_{\theta^n R})\le \theta^{-d} (\theta^n)^{2\alpha}\, \mathrm{Exc}(\nabla u;B_{R})\\
\le 
\theta^{-d-2\alpha} \left(\frac{r}{R}\right)^{2\alpha}\, \mathrm{Exc}(\nabla u;B_{R})\,,
\end{multline}
that is, \eqref{theorem 1 equation 2}. Notice that for $R\leq R^*(\delta)$ the bound \eqref{theorem 1 equation 2} is trivial up to choosing sufficiently large constant $C(d,a_\pm, \alpha)$.

The bound \eqref{theorem 1 equation 3} follows from \eqref{proposition 1 equation 4}, upon choosing the constant $\mathcal{C}$ in the definition of $r_*$ sufficiently large.

We complete the proof with the argument for \eqref{theorem 1 equation 4}.
For $r_*\le r \le R$, let $\xi_r\in \R^d$ be the vector defined by the relation
\begin{equation}
\mathrm{Exc}(\nabla u;B_r)=\fint_{B_r} \chi_\M |\nabla u -(\xi_r+\nabla\phi_{\xi_r})|^2
\end{equation}
whose existence is guaranteed by the nondegeneracy condition \eqref{theorem 1 equation 3}. The first step towards proving \eqref{theorem 1 equation 4} is to show that
\begin{equation}
\label{proof proposition 1 equation 8}
|\xi_\rho-\xi_{R'}|^2 \lesssim \mathrm{Exc}(\nabla u;B_{R'}) \qquad \text{for}\ \rho\le R'\le 2\rho, \quad r_*\le \rho\le R'\le R.
\end{equation}
By \eqref{theorem 1 equation 3} we have
\begin{multline}
\label{proof proposition 1 equation 9}
|\xi_\rho-\xi_{R'}|^2 \lesssim \fint_{B_\rho} \chi_\M|(\xi_\rho-\xi_{R'}+\nabla \phi_{\xi_\rho-\xi_{R'}})|^2
\\
=
\fint_{B_\rho} \chi_\M|(\xi_\rho+\nabla \phi_{\xi_\rho})-\nabla u+\nabla u-(\xi_{R'}+\nabla \phi_{\xi_{R'}})|^2
\\
\lesssim
\fint_{B_\rho} \chi_\M|\nabla u-(\xi_\rho+\nabla \phi_{\xi_\rho})|^2+\fint_{B_{R'}} \chi_\M|\nabla u-(\xi_{R'}+\nabla \phi_{\xi_{R'}})|^2
\\
=
\mathrm{Exc}(\nabla u;B_\rho)+ \mathrm{Exc}(\nabla u;B_{R'}) \,.
\end{multline}
Note that in the second-to-last step we used the fact that $\rho\sim R'$ so that $\fint_{B_\rho}\lesssim \fint_{B_{R'}}$. But \eqref{proof proposition 1 equation 9} implies \eqref{proof proposition 1 equation 8}.

Next we use a dyadic argument. For $N\in \N$ such that $2^{-(N+1)} R' < r \le 2^{-N} R'$ by  \eqref{proof proposition 1 equation 8} we have 
\begin{equation*}
	|\xi_r - \xi_{2^{-N}R'}|^2 \lesssim \mathrm{Exc}(\nabla u;B_{2^{-N} R'}), \quad  |\xi_{2^{-(n+1)}R'} - \xi_{2^{-n}R'}|^2 \lesssim \mathrm{Exc}(\nabla u;B_{2^{-n} R'}), \, n=0,1,\dots,N-1.
\end{equation*} 
Then from \eqref{theorem 1 equation 2} it follows that 
\begin{multline}\label{3.15}
	|\xi_r - \xi_{R'}|^2 \lesssim \left( \sum_{n=0}^N \sqrt{ \mathrm{Exc}(\nabla u;B_{2^{-n} R'})}\right)^2
	\\
	\lesssim \left( \sum_{n=0}^N2^{-n\alpha} \sqrt{ \mathrm{Exc}(\nabla u;B_{ R'})}\right)^2 \lesssim \mathrm{Exc}(\nabla u;B_{ R'})
\end{multline}  

Now the bound \eqref{theorem 1 equation 4} follows by applying the triangle inequality and \eqref{theorem 1 equation 3}, \eqref{theorem 1 equation 2} and \eqref{3.15}:
\begin{multline}
\label{proof proposition 1 equation 10}
\fint_{B_r}\chi_\M |\nabla u|^2  \le \mathrm{Exc}(\nabla u, B_r)+\fint_{B_r}\chi_\M |\xi_r+\nabla \phi_{\xi_r}|^2 
 \lesssim\mathrm{Exc}(\nabla u, B_r)+|\xi_r|^2
\\
 \le
\mathrm{Exc}(\nabla u, B_r)+|\xi_r-\xi_R|^2+|\xi_R|^2
 \lesssim
\mathrm{Exc}(\nabla u, B_R)+|\xi_R|^2
\lesssim \fint_{B_R}\chi_\M |\nabla u|^2\,.
\end{multline}
In the last step we used the chain of inequalities
$|\xi_R|^2 \lesssim  \fint_{B_R}\chi_\M|\nabla \phi_{\xi_R}+\xi_R|^2\le \mathrm{Exc}(\nabla u, B_R)+\fint_{B_R}\chi_\M |\nabla u|^2$ and the definition of excess, which is defined as an infimum.
\end{proof}

Next, we prove an analogue of Proposition~\ref{theorem 1} for degenerate elliptic equations with nonzero source term.

\begin{proposition}
\label{corollary 3}
Given $\alpha\in(0,1)$, let $r_*$ and $\tilde B_R$ be as in Proposition~\ref{theorem 1}.
Suppose that $u\in H^1_\loc(\R^d)$ and $g\in (L^2_\loc(\R^d))^d$, $h\in L^2_\loc(\R^d)$ are related as 
\begin{equation}
\label{corollary 3 equation 1}
-\nabla \cdot a(\nabla u +g)= \chi_\M h \qquad \text{in}\ \tilde B_{R}\,.
\end{equation}
Then we have
\begin{multline}
\label{corollary 3 equation 2}
\sup_{r\in[r_*, R]} \frac{1}{r^{2\alpha}}\mathrm{Exc}(\nabla u+g;B_r)
\lesssim \frac{1}{R^{2\alpha}} \mathrm{Exc}(\nabla u+g;B_R)
\\
+\sup_{r\in[r_*, R]} \frac{1}{r^{2\alpha}} \fint_{B_r} \chi_\M\left(  \left|g-\fint_{B_r\cap \M} g \right|^2 + r^2|h|^2\right)\,.
\end{multline}
Furthermore, we have
\begin{multline}
\label{corollary 3 equation 3}
\sup_{r\in[r_*, R]} \fint_{B_r}\chi_\M|\nabla u+g|^2
\lesssim \fint_{B_R}\chi_\M|\nabla u+g|^2
\\
+\sup_{r\in[r_*, R]} \left(\frac{R}{r}\right)^{2\alpha} \fint_{B_r} \chi_\M\left(  \left|g-\fint_{B_r\cap \M} g \right|^2 + r^2|h|^2\right)\,.
\end{multline}
Here $\lesssim$ means $\le C(d,a_\pm,\alpha)$.
\end{proposition}

\begin{proof}

Let $\alpha\in (0,1)$ and put $\alpha':=\frac{\alpha+1}{2}\in (\alpha,1)$. Let $r_*$ be the radius corresponding to the H\"older exponent $\alpha'$.
Let $r_*\le r\le \rho\le R$. We claim that
\begin{multline}
\label{proof corollary 3 equation 1}
\mathrm{Exc}(\nabla u+g;B_r)\le C_1(d,a_\pm,\alpha) \left[ \left(\frac{r}{\rho}\right)^{2\alpha'}\mathrm{Exc}(\nabla u+g;B_\rho)\right.
\\
\left.
+\left(\frac{\rho}{r} \right)^d \fint_{B_\rho} \chi_\M \left(  \left|g-\fint_{B_\rho\cap \M} g \right|^2 + \rho^2|h|^2\right)
\right].
\end{multline}
In order to show this, put $\zeta:=\fint_{B_{\rho}\cap \M} g$ and let $w$ be the solution of
\begin{equation}
\label{proof corollary 3 equation 2}
-\nabla \cdot a \nabla w= \nabla \cdot (a(g-\zeta))+\chi_\M h \quad \text{in}\ \tilde B_{\rho-1}, \quad w=0\quad \text{on}\ \partial \tilde B_{\rho-1},
\end{equation}
where $\tilde{B}_{\rho-1}$ is defined in the same way as $\tilde B_R$ from the statement of the proposition. It is not hard to see that \eqref{proof corollary 3 equation 2} implies that the function $x\mapsto u+\zeta\cdot x-w$ is $a$-harmonic in $\tilde B_{\rho}$. Hence, by Proposition~\ref{theorem 1} we have
\begin{equation}
\label{proof corollary 3 equation 3}
\mathrm{Exc}(\nabla u+\zeta-\nabla w;B_r)\lesssim \left(\frac{r}{\rho}\right)^{2\alpha'} \mathrm{Exc}(\nabla u+\zeta-\nabla w; B_{\rho-1})\,.
\end{equation}
Furthermore, an energy estimate for \eqref{proof corollary 3 equation 2} gives us\footnote{Here the argument goes as follows: test \eqref{proof corollary 3 equation 2} against $w^\ext$ in $\tilde B_{\rho-1}$, extend $w^\ext$ by zero in $B_{\rho}\setminus \tilde B_{\rho-1}$, and finally use the Poincar\'e inequality in $B_\rho$.}
\begin{equation}
\label{proof corollary 3 equation 4}
\fint_{B_{\rho}} \chi_\M |\nabla w|^2
\lesssim \fint_{B_{\rho}} \chi_\M \left(  \left|g-\fint_{B_\rho\cap \M} g \right|^2 + \rho^2|h|^2\right).
\end{equation}

Now, adding and subtracting $\zeta-\nabla w$, using triangle inequality  and the bounds \eqref{proof corollary 3 equation 3} and \eqref{proof corollary 3 equation 4} yields
\begin{multline}
\label{proof corollary 3 equation 5}
\mathrm{Exc}(\nabla u +g;B_r)=\inf_{\xi\in \R^d} \fint_{B_r} \chi_\M |\nabla u+g-(\xi+\nabla \phi_\xi)|^2
\\
\lesssim
\mathrm{Exc}(\nabla u +\zeta-\nabla w;B_r)+ \fint_{B_r}\chi_\M |g-\zeta|^2+\fint_{B_r}\chi_\M |\nabla w|^2
\\
\lesssim
\left(\frac{r}{\rho}\right)^{2\alpha'}\left[ \mathrm{Exc}(\nabla u +g;B_\rho)+ \fint_{B_\rho}\chi_\M(|g-\zeta|^2+|\nabla w|^2)\right]
\\
+\left(\frac{\rho}{r} \right)^d\fint_{B_{\rho}} \chi_\M \left(  \left|g-\fint_{B_\rho\cap \M} g \right|^2 + \rho^2|h|^2\right)\,.
\end{multline}
Utilising \eqref{proof corollary 3 equation 4} once again we arrive at \eqref{proof corollary 3 equation 1}.

Retracing the argument from~\cite[Proof of Corollary~3]{mjm} it is easy to see that \eqref{proof corollary 3 equation 1} implies \eqref{corollary 3 equation 2}.

Finally, let us prove \eqref{corollary 3 equation 3}.
First observe that \eqref{corollary 3 equation 2} yields the following: for all $r\ge  r_*$ one has
\begin{multline}
\label{proof corollary 3 equation 6}
\sup_{\rho\in[r,R]} \left(\frac{R}{\rho}\right)^{2\alpha}\mathrm{Exc}(\nabla u+g;B_\rho)
\lesssim  \mathrm{Exc}(\nabla u+g;B_R)
\\
+\sup_{\rho\in[r,R]} \left(\frac{R}{\rho}\right)^{2\alpha} \fint_{B_\rho} \chi_\M\left(  \left|g-\fint_{B_\rho\cap\M} g \right|^2 + \rho^2|h|^2\right)\,.
\end{multline}
Retracing the dyadic argument in the proof of Proposition~\ref{theorem 1} it is easy to see that
\begin{equation}
\label{proof corollary 3 equation 7}
|\xi_r-\xi_R|\lesssim \sup_{\rho\in[r,R]}\left(\frac{R}{\rho}\right)^{2\alpha}\mathrm{Exc}(\nabla u+g;B_\rho)\qquad \forall\alpha>0.
\end{equation}

Then using the triangle inequality and  the bounds \eqref{proof corollary 3 equation 7} and  \eqref{proof corollary 3 equation 6} we get
\begin{multline}
\label{proof corollary 3 equation 8}
|\xi_r|^2+\mathrm{Exc}(\nabla u+g;B_r)
 \lesssim
|\xi_R|^2+ \mathrm{Exc}(\nabla u+g;B_R)
\\
+\sup_{\rho\in[r,R]} \left(\frac{R}{r}\right)^{2\alpha} \fint_{B_\rho} \chi_\M\left(  \left|g-\fint_{B_\rho\cap\M} g \right|^2 + \rho^2|h|^2\right)\,.
\end{multline}
By combining \eqref{proof corollary 3 equation 8} with \eqref{proof proposition 1 equation 10}, replacing in the latter $\nabla u\mapsto \nabla u+g$, one arrives at \eqref{corollary 3 equation 3}.
\end{proof}

\bbb

Equipped with Proposition~\ref{corollary 3} we can now prove the following statement, where the massive corrector $\phi_T$ comes into the picture. By making a quantitative assumption on the growth of the massive corrector $\phi_T$ in terms of a ``new'' regularity radius $r_{**}$, we will quantify the sublinear growth of the (extended) corrector itself. Furthermore, we will show that $r_*$ is bounded above by $r_{**}$.

\begin{proposition}
\label{proposition 2}
Suppose that we have
\begin{equation}
\label{proposition 2 equation 1}
\fint_{B_R} \frac{1}{T}|(\phi_T^\ext,\sigma_T)|^2\le \left(\frac{r_{**}}{R}\right)^{2\nu}
\end{equation}
for some $\nu, r_{**}>0$ and all $R>r_{**}$, $R=2^k$, $T=R^2$.
There exists $C=C(d,a_\pm,\alpha,\nu)$
such that
\begin{equation}
\label{proposition 2 equation 2}
\frac{1}{R^2}\fint_{B_R} |(\phi^\ext,\sigma)-\fint_{B_R}(\phi^\ext,\sigma)|^2 \le C \left(\frac{r_{**}}{R}\right)^{2\nu} \qquad \text{for all}\quad R>r_{**}.
\end{equation}
Furthermore we have 
$$r_*\leq C\,r_{**},$$
possibly after enlarging $C$. 
\end{proposition}

\begin{proof}
This proof is very much based on \cite[Proof of Proposition 2]{mjm}. Actually, it follows verbatim except for the Step 3 in that proof, which requires the use of the PDE (which is obviously the only difference compared to our setting). Let us therefore focus on this step. 


{ \color{black}

\medskip

To show the analogous statement as in \cite[Step 3 in the Proof of Proposition 2]{mjm}, after application of the Poincar\'e inequality we need to show the following
\begin{multline}
\label{proof of proposition 2 equation 5}
\fint_{B_r \cap \M} |\nabla (\phi_{T_0}-\phi_T)|^2
\lesssim
\left(
\frac{1}{R^2}+ \left( \frac{R}{r}\right)^d\frac1{T_0}
 \right)
  \inf_c \fint_{B_R \cap \M}(\phi_{T_0}-\phi_T-c)^2
  \\
  +
   \left( \frac{R}{r}\right)^{d}\fint_{B_R \cap \M}\frac1T(\phi_{T})^2
\end{multline}
for $r_*\le r\le \frac14 R \le \sqrt{T_0}$ with $T = R^2$. 

To prove \eqref{proof of proposition 2 equation 5}, observe that by the definition of the massive correctors $\phi_T$ and $\phi_{T_0}$
\eqref{massive corrector}, their difference $u := \phi_{T_0} - \phi_T$
solves 
\begin{equation}
\label{proof of proposition 2 equation 6}
\frac1{T_0}\chi_\M u-\nabla \cdot a\nabla u=\chi_\M f
\end{equation}
with $f = \left( \frac 1T - \frac 1{T_0}\right) \phi_T$. 

Testing \eqref{proof of proposition 2 equation 6} with $\eta^2(u^\ext-\bar u)$, where $\eta$ is a cutoff of $B_{R/2}$ in $B_{R}$ and $\bar u:= \frac{\int \eta^2 \chi_\M u}{\int \eta^2 \chi_\M}$ we get
\begin{multline}
\label{proof of proposition 2 equation 6a}
\int \frac1{T_0} \chi_\M \eta^2(u^\ext-\bar u)^2+ \int a\nabla u^\ext\cdot \nabla( \eta^2(u^\ext-\bar u))
=\int \chi_\M f \eta^2(u^\ext-\bar u).
\end{multline}
In order to get $(u^\ext-\bar u)^2$ on the LHS the above choice of $\bar u$ was crucial. Applying Young's inequality on the RHS and absorbing the gradient term into the LHS we get for $\alpha > 0$ large enough 
\begin{multline}
\label{proof of proposition 2 equation 7}
\int_{B_R\cap \M} \left( \frac1{T_0} \eta^2(u-\bar u)^2+ |\nabla(\eta(u -\bar u))|^2\right)\\
\lesssim \alpha^{-1}\int_{B_R\cap \M} R^2 \textcolor{black}{\eta^2}  f^2 
+ (\alpha+1) \int_{B_R\cap\M}\frac1{R^2}(u-\bar u)^2\,,
\end{multline}
where we also used that all the terms were vanishing outside of $\M$, hence we could have replaced $B_R$ with $B_R \cap \M$ (while removing $\chi_\M$ as well as the need for extension). 
%
On the other hand, testing \eqref{proof of proposition 2 equation 6} with $\eta^2$ yields
\begin{equation}
\label{proof of proposition 2 equation 9}
\int\frac1{T_0}\chi_\M u^\ext\eta^2 +a\nabla u^\ext \cdot \nabla(\eta^2)=\int\chi_\M f \eta^2,
\end{equation}
which, combined with the identity
\begin{equation}
\label{proof of proposition 2 equation 10}
a\nabla u \cdot \nabla(\eta^2)=2 a\nabla(\eta(u-\bar u))\cdot \nabla \eta-2 (u-\bar u) a\nabla \eta\cdot \nabla \eta
\end{equation}
implies
\begin{equation}
\label{proof of proposition 2 equation 11}
R^2 \left(\frac1{T_0} \bar u \right)^2 \lesssim \fint_{B_R\cap \M} \left(\frac{1}{R^2}(u-\bar u)^2 + R^2\eta^2 f^2+ |\nabla (\eta(u-\bar u))|^2 \right).
\end{equation}
To get $\bar u$ on the LHS we used that $\int \eta^2 \chi_\M \gtrsim R^d$, i.e., that by removing holes we take away only portion of the volume. 

By combining \eqref{proof of proposition 2 equation 7}, \eqref{proof of proposition 2 equation 11} we arrive at
\begin{multline}
\label{proof of proposition 2 equation 12}
R^2 \left(\frac1{T_0} \bar u \right)^2 +\fint_{B_{R/2}\cap \M} \left( \frac1{T_0} (u -\bar u)^2+ |\nabla u|^2\right)
\\
\lesssim
\fint_{B_R \cap \M} \left(\frac{1}{R^2}(u-\bar u)^2 + \chi_\M R^2 \eta^2 f^2\right).
\end{multline}
We further claim that 
\begin{equation}
\fint_{B_R} \frac{1}{R^2}\chi_\M (u-\bar u)^2 \lesssim \inf_c \fint_{B_R \cap \M} \frac{1}{R^2}(u-c)^2\,.
\end{equation}
This follows from Lemma~\ref{lemma about contradiction argument}, modulo rescaling. Observe that the constant hidden inside $\lesssim$ is independent of $R$ and perforation $\omega$. 

Therefore, \eqref{proof of proposition 2 equation 12} can be equivalently rewritten as
\begin{multline}
\label{proof of proposition 2 equation 12a}
R^2 \left(\frac1{T_0} \bar u \right)^2 +\fint_{B_{R/2}\cap \M} \left( \frac1{T_0} (u-\bar u)^2+ |\nabla u|^2\right)
\lesssim
\inf_c \fint_{B_R\cap\M} \left(\frac{1}{R^2}(u-c)^2 + R^2 f^2\right).
\end{multline}

Let $v$ and $w$ be the solutions to the auxiliary boundary value problems 
\begin{equation}
\label{proof of proposition 2 equation 13}
-\nabla \cdot a \nabla v=-\frac{1}{T_0}\chi_\M \bar u \qquad \text{in } \tilde B_{R/2}, \quad v=u\ \text{on}\ \partial \tilde B_{R/2},
\end{equation}
\begin{equation}
\label{proof of proposition 2 equation 14}
-\nabla \cdot a \nabla w=-\frac{1}{T_0}\chi_\M (u-\bar u)+\chi_\M f \qquad \text{in }\tilde B_{R/2}, \quad w=0\ \text{on}\ \partial \tilde B_{R/2}\,.
\end{equation}
Observe that $- \nabla \cdot a \nabla u = - \nabla \cdot a \nabla(v+w)$ in $\tilde B_{R/2}$ and $u=v+w$ agree on the boundary, hence $u=v+w$. 

Let us now get estimates on $v$ and $w$. Testing \eqref{proof of proposition 2 equation 13} with $v-u$ we obtain
\begin{equation}
\label{proof of proposition 2 equation 13+}
\int_{\tilde B_{R/2}} a \nabla v \cdot \nabla (v-u)=-\frac{1}{T_0} \bar u \int_{\tilde B_{R/2}} \chi_\M (v-u)\,,
\end{equation}
which after Cauchy-Schwarz and averaging integrals becomes 
\begin{equation}
\label{proof of proposition 2 equation 15}
\fint_{\tilde B_{R/2}\cap \M} |\nabla v|^2 \lesssim \fint_{\tilde B_{R/2}\cap \M} |\nabla u|^2 + \alpha^{-1}R^2\left(\frac{1}{T_0} \bar u\right)^2+ \frac{\alpha}{R^2}\fint_{\tilde B_{R/2}\cap \M} |v-u|^2\,.
\end{equation}

Since $u-v=0$ on $\partial \tilde B_{R/2}$, we can simply extend $u-v$ by zero outside and apply Poincar\'e inequality in $B_{R/2+1}$ to the effect of
\begin{equation}
\label{proof of proposition 2 equation 16}
\int_{\tilde B_{R/2}\cap \M} |v-u|^2\lesssim R^2 \int_{\tilde B_{R/2}\cap \M} |\nabla(v-u)|^2\,.
\end{equation}
Substituting \eqref{proof of proposition 2 equation 16} into \eqref{proof of proposition 2 equation 15}, and reabsorbing the gradient of $v$ into the LHS by choosing $\alpha$ sufficiently small, we thus obtain
\begin{equation}
\label{proof of proposition 2 equation 18}
\fint_{\tilde B_{R/2}\cap \M} |\nabla v|^2 \lesssim \fint_{\tilde B_{R/2}\cap \M} |\nabla u|^2 +R^2\left(\frac{1}{T_0} \bar u\right)^2\,.
\end{equation}

We turn now to $w$. By testing \eqref{proof of proposition 2 equation 14} with $w$, which we can since $w$ is zero on the boundary,  we obtain
\begin{equation}
\label{proof of proposition 2 equation 19}
\fint_{\tilde B_{R/2}\cap \M} |\nabla w|^2 \lesssim \fint_{\tilde B_{R/2}\cap \M} \left(\alpha^{-1}\frac{R^2}{T_0^2}|u-\bar u|^2+ R^2 f^2 \right)+ \fint_{\tilde B_{R/2}\cap \M} \frac\alpha{R^2} w^2\,.
\end{equation}
By the same argument as above for $u-v$ we have that 
\begin{equation}
\label{proof of proposition 2 equation 20}
\int_{\tilde B_{R/2}\cap \M} w^2\lesssim R^2 \int_{\tilde B_{R/2}\cap \M} |\nabla w|^2\,,
\end{equation}
where the constant in $\lesssim$ is independent of $\M$. Thus, \eqref{proof of proposition 2 equation 20} leads to
\begin{equation}
\label{proof of proposition 2 equation 21}
\fint_{\tilde B_{R/2}\cap \M} |\nabla w|^2 \lesssim \fint_{\tilde B_{R/2}\cap \M} \left(\frac{R^2}{T_0^2}|u-\bar u|^2 + R^2 f^2\right)\,.
\end{equation}

Now we can retrace once again the arguments from~\cite[Proof of Proposition 2]{mjm}. To get~\eqref{proof of proposition 2 equation 5}, we employ Schauder's Estimates Proposition~\ref{corollary 3} in form of~\eqref{corollary 3 equation 3}. 
Observing that $h$ satisfies
\begin{equation}\nonumber
  \sup_{1 \le r \le R} \inf_{\xi \in \R^d} \biggl( \frac Rr \biggr)^{2\alpha} r^2 \fint_{B_r \cap\M} \left|\frac{1}{T_0} \bar u \right|^2 \overset{\alpha \le 1}{\lesssim} R^2 \biggl( \frac{1}{T_0} \bar u \biggr)^2\,,
\end{equation}
we obtain from \eqref{corollary 3 equation 3} as well as \eqref{proof of proposition 2 equation 18}
that for $r_* \le r \le \frac R2$ 
\begin{equation}\nonumber
    \fint_{B_r \cap\M} |\nabla v|^2 \lesssim R^2 \biggl( \frac{1}{T_0} \bar u \biggr)^2 + \fint_{\tilde B_{R/2} \cap\M} |\nabla u|^2. 
\end{equation}
For $w$ we simply use \eqref{proof of proposition 2 equation 21} and monotonicity of integral so that for $r_* \le r \le \frac R2$ 
\begin{equation}
    \fint_{B_r \cap\M} |\nabla w|^2 \lesssim \biggl( \frac Rr \biggr)^{d} \fint_{\tilde B_{R/2}\cap \M} \left(\frac{R^2}{T_0^2}|u-\bar u|^2 + R^2 f^2\right)\,.
\end{equation}

Since $u=v+w$, combining the previous estimates on $v$ and $w$ by triangle inequality together with~\eqref{proof of proposition 2 equation 12a} we get for $r_* \le r \le \frac R2$
\begin{equation}
  \fint_{B_r\cap\M} |\nabla u|^2 \overset{\text{Lemma}~\ref{lemma about contradiction argument}}{\lesssim} \biggl(\frac1{R^2} + \biggl( \frac Rr\biggr)^d \frac 1{T_0} \biggr) \inf_c \fint_{B_R \cap \M} (u-c)^2 + \biggl( \frac Rr\biggr)^d R^2 \fint_{B_R \cap \M} f^2.
\end{equation}
With our choice  $u = \phi_{T_0} - \phi_T$ and $f = \left( \frac 1T - \frac 1{T_0}\right) \phi_T$ with $T \le 16T_0$ the previous estimate yields \eqref{proof of proposition 2 equation 5}. 
Since the matrix $\M$ plays no role in the definition of the flux corrector $\sigma$, one gets to the analogous estimate~\cite[(163)]{mjm}. 

Since the rest of the argument does not use equation at all but instead builds on the acquired one-step estimates, let us just quickly sketch it (for details see Steps 1 to 3 in~\cite[Proof of Proposition 2]{mjm}). 

First, a classical Campanato iteration argument~\cite[Step 2, Proof of Proposition 2]{mjm} turns one-step estimates into
\begin{equation}\nonumber
    \frac{1}{R^2} \inf_c \fint_{B_R \cap \M} |(\phi_T,\sigma_T)-c|^2 \lesssim \biggl( \frac{r_{**}}{R}\biggr)^{2\nu}
\end{equation}
    for $\max\{r_*,r_{**}\}$$ \le R \le \sqrt{T_0}$ and $R^2=T$. Second, using reduction argument and the fact that by qualitative homogenization the  differences $\phi_T - \phi$ as well as $\sigma_T - \sigma$ measured in averaged sense in $L^2(B_{R})$ converge to $0$, we complete the proof of this Proposition. }
\end{proof}

Next, in Proposition~\ref{proposition 3}, we show that the massive corrector $\phi_T$ can be controlled by the average of the $H^1$-norm of field and flux of the massive corrector $\phi_t$ localised on a smaller scale ($t\le T$).

\

Let
\begin{equation}
\label{exponential weight}
\omega_T(x):=\frac{1}{|\partial B_1|(d-1)!}\frac{1}{(\kappa\sqrt{T})^d}\exp\left(-\frac{|x|}{\kappa\sqrt{T}} \right)\,,
\end{equation}
where $\kappa$ is a positive constant.

\begin{proposition}
\label{proposition 3}
For all $0<t\le T$ we have
\begin{equation}
\label{propositon 3 equation 1}
\int \omega_T \frac1T |(\phi_T^\ext,\sigma_T)|^2\le C(d,a_\pm) \int\omega_T \left(\frac1t |\phi_t^\ext|^2+\frac1t |g_t|^2+|\nabla g_t|^2  \right)\,,
\end{equation}
where $g_T$ is vector field defined as the solution of
\begin{equation}
\label{equation definition of g_T}
\frac1T g_T-\Delta g_T=\frac{1}{\sqrt{T}}(q_T-\E[q_T])\,.
\end{equation}
\end{proposition}

\begin{proof}
    We observe that it is enough to prove the estimate with $\phi_T^\ext$ and $\phi_t^\ext$ replaced by $\chi_\M\phi_T$ and $\chi_\M\phi_t$, respectively. For this, the proof retraces, modulo very minor changes, \cite[Proof of Proposition~3]{mjm}. At first sight, the gradient of the convolution of $\phi_t$ could be problematic in the presence of perforations. However, one observes that the aforementioned convolution is never evaluated at $\tau=0$ ($\tau$ being the convolution parameter). Indeed, in the argument from \cite[Proof of Proposition~3]{mjm} one can always handle the terms containing the gradient of the convolution by discharging the derivatives onto the mollifier, and estimating the extra terms away. Hence, effectively one only needs to work with the convolution of $\phi_t$, which is unproblematic.
\end{proof}

Now, in preparation for the next proposition, we state and prove two technical lemmata. The first establishes almost sure bounds on the local $H^1$ norm of the massive corrector $\phi_T$, whereas the second quantifies, under suitable assumptions, the exponential growth of the $H^1$ norm of solutions of degenerate elliptic equations in divergence form in the presence of a massive term via the introduction of an appropriate exponential weight.

\begin{lemma}
\label{lemma corrector bounded on balls}
The massive corrector $\phi_T$ defined in accordance with \eqref{massive corrector} satisfies the estimate
\begin{equation}
\label{lemma corrector bounded on balls equation 1}
\int \omega_T \left(\frac1T (\phi_T^\ext)^2+|\nabla \phi_T^\ext|^2\right)<+\infty
\end{equation}
almost surely.
\end{lemma}

\begin{proof}
Define the family of random variables
\begin{equation}
\label{proof lemma corrector bounded on balls equation 2}
X_z:=\int_{\square_1(z)} \left(\frac1T (\phi_T^\ext)^2+|\nabla \phi_T^\ext|^2\right), \qquad z\in \mathbb{Z}^d.
\end{equation}
Clearly, by stationarity, 
\begin{equation}
\label{proof lemma corrector bounded on balls equation 2bis}
\E  X_z=\E X_0\qquad \text{for all }z\in \Z^d.
\end{equation}
By means of a standard discretisation argument, we can write
\begin{equation}
\label{proof lemma corrector bounded on balls equation 3}
\int \omega_T \left(\frac1T (\phi_T^\ext)^2+|\nabla \phi_T^\ext|^2\right)=\sum_{z\in \mathbb{Z}} w_z  X_z,
\end{equation}
for some family of (exponential) weights $w_z$. It is easy to see that there exists a family of positive constants $\{c_z\}_{z\in \mathbb{Z}}$ such that
\begin{equation}
\label{proof lemma corrector bounded on balls equation 4}
\sum_{z\in \mathbb{Z}}\frac1{c_z}<+\infty \qquad \text{and}\qquad \sum_{z\in \mathbb{Z}}c_z w_z<+\infty.
\end{equation}
For instance, one can take $c_z=w_z^{-1/2}$.

With account of formulae \eqref{proof lemma corrector bounded on balls equation 2bis}, \eqref{proof lemma corrector bounded on balls equation 4}, the Borel--Cantelli Lemma and Markov's Inequality ensure that $X_z>c_z$ for at most finitely many $z$'s. Hence 
\eqref{proof lemma corrector bounded on balls equation 1} follows from
\eqref{proof lemma corrector bounded on balls equation 3} and \eqref{proof lemma corrector bounded on balls equation 4}.
\end{proof}

\begin{remark}
Let us emphasise that the estimate~\eqref{lemma corrector bounded on balls equation 1} is not uniform in $\omega$, but it is `only' an almost sure bound, as the proof shows. Uniformity will be a consequence of the proof of the upcoming lemma.
\end{remark}

\begin{lemma}
\label{prop:omega_T estimate}
Let $f\in L^2_\mathrm{loc}(\R^d)$ and $h\in [L^2_\mathrm{loc}(\R^d)]^d$ be such that
$$
\int_{\R^d}\omega_T\left(\frac1T f^2+|h|^2 \right)<+\infty\,.
$$ 
Suppose $\phi$ belongs to the class of functions with the property
\begin{equation*}
\label{prop:omega_T estimate equation 1}
\int_{\M} \omega_T \left(\frac1T \phi^2+|\nabla \phi|^2\right)<+\infty \qquad \text{a.s.}
\end{equation*}
and satisfies
\begin{equation}
\label{29 January 2024 equation 1}
\frac1T\int_{\R^d} \chi_\M \phi\, v +\int_{\R^d}a\, \nabla \phi \cdot  \nabla v=\frac1T\int_{\R^d} fv+ \int_{\R^d} h \cdot \nabla v
\end{equation} 
for all $v\in H^1(\R^d)$ with compact support.
Then we have
\begin{equation}
\label{29 January 2024 equation 2}
\int_{\R^d} \omega_T \left( \frac1T(\phi^\ext)^2 +|\nabla \phi^\ext|^2\right)\lesssim
\int_{\R^d} \omega_T \left( \frac1T f^2 +|h|^2\right)\,.
\end{equation} 
\end{lemma}

\begin{proof}
    Let $\phi^\ext$ be the extension of $\phi|_\M$ into $\omega$. Let $\psi$ be a positive, smooth, and sufficiently fast-decaying function. Let us (formally, for now) test \eqref{29 January 2024 equation 1} with $\psi^2\phi^\ext \in H^1(\R^d)$: this step will be justified \emph{a posteriori} for a specific choice of function $\psi$. Using the fact that $\chi_\M\phi=\chi_\M \phi^\ext$ we obtain
\begin{equation}
\label{29 January 2024 equation 3}
\frac1T\int_{\R^d} \chi_\M \psi^2 (\phi^\ext)^2+\int_{\R^d}a\, \nabla \phi^\ext \cdot  \nabla (\psi^2\phi^\ext)=\frac1T\int_{\R^d} \psi^2 f\phi^\ext+ \int_{\R^d} h \cdot \nabla (\psi^2\phi^\ext) \,.
\end{equation} 
Purely algebraic manipulations give us
\begin{equation}
\label{29 January 2024 equation 4}
a\, \nabla \phi^\ext \cdot  \nabla (\psi^2\phi^\ext)=a\, \nabla (\psi\phi^\ext) \cdot  \nabla (\psi\phi^\ext) - (\phi^\ext)^2 a \nabla \psi \cdot \nabla \psi\,.
\end{equation}
Since
 $h \cdot \nabla (\psi^2\phi^\ext) = h \cdot \left[ \psi\nabla (\psi\phi^\ext) +\psi\phi^\ext\nabla\psi \right]$,
substituting \eqref{29 January 2024 equation 4} into \eqref{29 January 2024 equation 3} and using Young's inequality we obtain
\begin{multline}
\label{29 January 2024 equation 6}
\int_{\M}  \left(\frac1T \psi^2 (\phi^\ext)^2+|\nabla (\psi\phi^\ext)|^2\right)
\\
\lesssim
\int_{\R^d}(\phi^\ext)^2 a \nabla \psi \cdot \nabla \psi
+
\frac1T\int_{\R^d} \psi^2 f\phi^\ext+ \int_{\R^d} h \cdot \nabla (\psi^2\phi^\ext)
\\
\lesssim 
(1+\beta) \int_{\R^d}|\nabla \psi|^2(\phi^\ext)^2
+
\frac{\gamma^{-1}}T\int_{\R^d} \psi^2 f^2
\\
+
\frac\gamma T \int_{\R^d} \psi^2 (\phi^\ext)^2
+
\alpha \int_{\R^d} |\nabla\left(\psi \phi^\ext\right)|^2
+
(\alpha^{-1}+\beta^{-1})\int_{\R^d} \psi^2|h|^2\,.
\end{multline}
Here $\alpha,\beta, \gamma$ are positive constants that can be chosen arbitrarily and come from Young's inequality. 

Before proceeding, we observe that we can replace $|\nabla(\psi\phi^\ext)|^2$ with $\psi^2|\nabla\phi^\ext|^2$ everywhere in \eqref{29 January 2024 equation 6} up to adjusting the constants. Indeed, we have $|\nabla(\psi\phi^\ext)|^2=\psi^2|\nabla\phi^\ext|^2+ |\nabla\psi|^2(\phi^\ext)^2 + 2 \psi\phi^\ext \nabla \psi \cdot \nabla \phi^\ext$, and the term $2 \psi\phi^\ext \nabla \psi \cdot \nabla \phi^\ext$ can be reabsorbed by the remaining terms using Young's inequality once more. 

Let us now put $\psi:=\sqrt{T}^{\frac{d}2}\omega_{\frac{T}2}$. Note that we have $\psi^2 \sim \omega_T$ and, upon choosing a sufficiently large constant $\kappa$ in the definition of $\omega_T$ (recall~\eqref{exponential weight}), $|\nabla \psi|^2 \ll \frac1T \omega_T$.
It is not difficult to see that for this choice of function $\psi$, all integrals appearing in the first part of the proof are severally finite. All steps above are then justified by a standard approximation argument of $\psi^2 \phi^\ext$ by compactly supported functions.

We claim that
\begin{equation}
\label{29 January 2024 equation 8}
\int_{\R^d} \omega_T (\phi^\ext)^2 \lesssim \int_{\M} \omega_T \phi^2
\end{equation}
and
\begin{equation}
\label{29 January 2024 equation 9}
\int_{\R^d} \omega_T |\nabla\phi^\ext|^2 \lesssim \int_{\M} \omega_T |\nabla\phi|^2\,.
\end{equation}
The estimates \eqref{29 January 2024 equation 8} and \eqref{29 January 2024 equation 9} are a straightforward consequence of the properties of the extension and the following simple observation. Let $\omega^k$ be a given inclusion and let $\mathcal{B}_\omega^k$ be the corresponding extension domain. Then we have\footnote{Here $C$ is the constant appearing in the definition of $\omega_T$, see~\eqref{exponential weight}.}
\begin{equation}
\label{29 January 2024 equation 10}
\omega_T(x) \le \sup_{y\in \mathcal{B}_\omega^k} \omega_T(y) \le e^{\sqrt{d}/(\kappa\sqrt{T})}\inf_{y\in \mathcal{B}_\omega^k} \omega_T(y)\, \qquad \forall x\in \mathcal{B}_\omega^k.
\end{equation}

Hence, by choosing Young's constants appropriately in \eqref{29 January 2024 equation 6} and using \eqref{29 January 2024 equation 8}--\eqref{29 January 2024 equation 9}, we arrive at \eqref{29 January 2024 equation 2}.
\end{proof}

A closer examination of the above proof gives us the following stronger property than what claimed in Lemma~\ref{lemma corrector bounded on balls} for the massive corrector $\phi_T$.
\begin{corollary}
    \label{corollary corrector bounded on balls}
The massive corrector $\phi_T$ defined in accordance with \eqref{massive corrector} satisfies the estimate
\begin{equation}
\label{corollary corrector bounded on balls equation 1}
\sup_{y\in \R^d} \int_{B_1(y)} \left(\frac1T (\phi_T^\ext)^2+|\nabla \phi_T^\ext|^2\right) < +\infty
\end{equation}
almost surely.
\end{corollary}

\begin{proof}
Let us observe that it is enough to establish the estimate
\begin{equation}
\label{proof lemma corrector bounded on balls equation 1}
\int \omega_T \left(\frac1T (\phi_T^\ext)^2+|\nabla \phi_T^\ext|^2\right)<M \qquad \text{a.s.},
\end{equation}
where $\omega_T$ is the exponential weight \eqref{exponential weight} and $M$ is a positive constant which is independent of $\omega$. Indeed, \eqref{proof lemma corrector bounded on balls equation 1} immediately implies the estimate
\begin{equation}
\label{proof lemma corrector bounded on balls equation 1bis}
\int_{B_1} \left(\frac1T (\phi_T^\ext)^2+|\nabla \phi_T^\ext|^2\right) < +\infty \qquad \text{a.s.}
\end{equation}
for a ball centred at the origin.
Then, one can consider equation \eqref{massive corrector} with coefficients translated by $y$ and obtain an estimate of the form~\eqref{proof lemma corrector bounded on balls equation 1bis} for the corresponding solution. Since $M$ is unaffected by translations, using the stationarity of $\phi_T$ one obtains the claim.

But the estimate~\eqref{proof lemma corrector bounded on balls equation 1} is a straightforward consequence of the proof of Lemma~\ref{prop:omega_T estimate} for the choice of $f=0$ and $h=a\mathbf{e}$.
\end{proof}

Now comes a crucial result. 
It tells us that the $F_T$ is $\sqrt{T}$-local in the sense of \eqref{31 January 2024 prop equation 2} and will enable us to apply  Lemma \ref{lemma 4}. 

\bbb

\begin{proposition}
\label{prop:F_t local}
For $T>0$ let
\begin{equation}
\label{31 January 2024 prop equation 1}
F_T(a):=\int_{\R^d\cap \M} \omega_T \left(\frac{1}{T}\phi_T^2+|\nabla \phi_T|^2+\frac1T g_T +|\nabla g_T|^2 \right),
\end{equation}
where $g_T$ is defined in accordance with~\eqref{equation definition of g_T}.
Then
\begin{equation}
\label{31 January 2024 prop equation 2}
a_1=a_2 \quad \text{in}\quad B_R \quad \Rightarrow \quad |F_T(a_1)-F_T(a_2)|\lesssim e^{-\frac{R}{C\sqrt{T}}}\,.
\end{equation}
\end{proposition}

\begin{proof}
Let $\M_1$ and $\M_2$ be two perforations associated with coefficients $a_1$ and $a_2$, respectively, and let $\phi_1$ and $\phi_2$ be the corresponding correctors. Namely, $\phi_j$, $j=1,2$, satisfy
\begin{equation}
\label{31 January 2024 equation 1}
\frac1T\int_{\R^d} \chi_{\M_j}\phi_j^{\ext,j} v+\int_{\R^d} a_j (\nabla \phi_j^{\ext,j}+ \mathbf{e})\cdot\nabla v=0
\end{equation}
for all $v\in H^1(\R^d)$ with compact support. Here by $(\,\cdot\,)^{\ext,j}$ we mean the extension of $(\,\cdot\,)|_{\M_j}$ into $\omega_j$.
For the sake of brevity, in what follows we will prove the ``reduced'' estimate
\begin{equation}
\label{31 January 2024 equation 0}
\left|\int_{\R^d\cap \M} \omega_T \left(\frac{1}{T}(\phi_1^2-\phi_1^2)+(|\nabla \phi_1|^2-|\nabla \phi_2|^2)\right)\right|
\lesssim
e^{-\frac{R}{C\sqrt{T}}}\,.
\end{equation}
The strategy to estimate the contributions to \eqref{31 January 2024 prop equation 1}, \eqref{31 January 2024 prop equation 2} coming from $g_T$ is analogous (in fact, easier) and left for the reader --- see also \cite[Proof of Proposition~4]{mjm}.

\

Formula \eqref{31 January 2024 equation 1}, Lemma~\ref{lemma corrector bounded on balls}, and Lemma~\ref{prop:omega_T estimate} imply
\begin{equation}
\label{31 January 2024 equation 2}
\int_{\R^d}\omega_T \left(\frac1T (\phi_j^{\ext,j})^2+|\nabla\phi_j^{\ext,j}|^2 \right)<+\infty
\end{equation}
almost surely. Let $v\in H^1(\R^d)$ with compact support. Then, by \eqref{31 January 2024 equation 1}, we have
\begin{multline}
\label{31 January 2024 equation 3}
\frac1T\int_{\R^d} \chi_{\M_1} \left(\phi_1^{\ext,1}-\phi_2^{\ext,2}\right) v+\int_{\R^d}a_1\left( \nabla\phi_1^{\ext,1}-\nabla\phi_2^{\ext,2}\right)\nabla v
\\
=
-\frac1T\int_{\R^d} (\chi_{\M_1}-\chi_{\M_2}) \phi_2^{\ext,2} v-\int_{\R^d}(a_1-a_2)\left( \nabla\phi_2^{\ext,2}+\mathbf{e}\right)\nabla v.
\end{multline}
Applying Lemma~\ref{prop:omega_T estimate} to \eqref{31 January 2024 equation 3} we obtain
\begin{multline}
\label{31 January 2024 equation 4}
\int_{\R^d}\omega_T \left(\frac1T \left[ \left(\phi_1^{\ext,1}-\phi_2^{\ext,2}\right)^{\ext,1} \right]^2+\left|\left( \nabla\phi_1^{\ext,1}-\nabla\phi_2^{\ext,2}\right)^{\ext,1} \right|^2 \right)
\\
\lesssim
\frac1T\int_{\R^d}\omega_T (\chi_{\M_1}-\chi_{\M_2})^2 (\phi_2^{\ext,2})^2
+
\int_{\R^d}\omega_T (a_1-a_2)^2\left|\nabla\phi_2^{\ext,2}+\mathbf{e} \right|^2\,.
\end{multline}
Next we observe that, clearly,
\begin{equation}
\label{31 January 2024 equation 5}
\omega_T(x)\lesssim e^{-\frac{|x|}{2\kappa\sqrt{T}}}\omega_{4T}.
\end{equation}
Moreover, by arguing as in the proof of Lemma~\ref{lemma corrector bounded on balls}, we obtain that
\begin{equation}
\label{31 January 2024 equation 6}
\int_{\R^d}\omega_{4T} \left(\frac1T (\phi_j^{\ext,j})^2+|\nabla\phi_j^{\ext,j}+\mathbf{e}|^2 \right)\lesssim 1.
\end{equation}
Therefore, if $a_1=a_2$ (and hence $\chi_{\M_1}=\chi_{\M_2}$) in $B_R:=B_R(0)$, formulae \eqref{31 January 2024 equation 4}--\eqref{31 January 2024 equation 6} imply
\begin{multline}
\label{31 January 2024 equation 8}
\int_{\R^d}\omega_T \left( \frac1T\left[ \left(\phi_1^{\ext,1}-\phi_2^{\ext,2}\right)^{\ext,1} \right]^2+\left|\left( \nabla\phi_1^{\ext,1}-\nabla\phi_2^{\ext,2}\right)^{\ext,1} \right|^2 \right)
\\
\lesssim
e^{-\frac{R}{2\kappa\sqrt{T}}}\,.
\end{multline}
Now, the two extensions $(\,\cdot\,)^{\ext,1}$ and $(\,\cdot\,)^{\ext,2}$ coincide in $B_R$, hence the LHS of \eqref{31 January 2024 equation 11} can be equivalently recast as
\begin{multline}
\label{31 January 2024 equation 8bis}
\int_{\R^d}\omega_T \left( \frac1T\left[ \left(\phi_1^{\ext,1}-\phi_2^{\ext,2}\right)^{\ext,1} \right]^2+\left|\left( \nabla\phi_1^{\ext,1}-\nabla\phi_2^{\ext,2}\right)^{\ext,1} \right|^2 \right)
\\
=
\int_{B_R}\omega_T \left( \frac1T\left(\phi_1^{\ext,1}-\phi_2^{\ext,2}\right)^2+\left|\nabla\phi_1^{\ext,1}-\nabla\phi_2^{\ext,2}\right|^2 \right)
\\
+
\int_{\R^d\setminus B_R}\omega_T \left( \frac1T\left[ \left(\phi_1^{\ext,1}-\phi_2^{\ext,2}\right)^{\ext,1} \right]^2+\left|\left( \nabla\phi_1^{\ext,1}-\nabla\phi_2^{\ext,2}\right)^{\ext,1} \right|^2 \right)\,.
\end{multline}
This justifies the upcoming definition, and the fact that we pursue two different strategies for proving the sought-after estimate in $B_R$ and in $\R^d\setminus B_R$.

Define 
\begin{equation}
\label{31 January 2024 equation 9}
F_{T,A}(a_j):=\int_A \omega_T \left( \frac1T \phi_j^2+|\nabla \phi_j|^2\right) \chi_{\M_j}\,.
\end{equation}
Then, in view of \eqref{31 January 2024 equation 8bis}, inside $B_R$ the estimates \eqref{31 January 2024 equation 8} and \eqref{31 January 2024 equation 2} give us
\begin{equation}
\label{31 January 2024 equation 10}
|F_{T,B_R}(a_1)-F_{T,B_R}(a_2)|\lesssim e^{-\frac{R}{4\kappa\sqrt{T}}}\,.
\end{equation}
Outside of $B_R$ we simply estimate the modulus of difference by the sum of the moduli, so that, using once again \eqref{31 January 2024 equation 5} and \eqref{31 January 2024 equation 6}, we obtain
\begin{multline}
\label{31 January 2024 equation 11}
|F_{T,\R^d\setminus B_R}(a_1)-F_{T,\R^d \setminus B_R}(a_2)|
\\
\le F_{T,\R^d\setminus B_R}(a_1)+F_{T,\R^d \setminus B_R}(a_2)
\\
\lesssim
 e^{-\frac{R}{2C\sqrt{T}}}
\int_{\R^d}\omega_{4T} \left(\frac1T (\phi_1^{\ext,1})^2+|\nabla\phi_1^{\ext,1}|^2 \right)
\\
+
 e^{-\frac{R}{2C\sqrt{T}}}
\int_{\R^d}\omega_{4T} \left(\frac1T (\phi_2^{\ext,2})^2+|\nabla\phi_2^{\ext,2}|^2 \right)
\\
\lesssim
 e^{-\frac{R}{2\kappa\sqrt{T}}}\,.
\end{multline}

By combining \eqref{31 January 2024 equation 10} and \eqref{31 January 2024 equation 11} we arrive at \eqref{31 January 2024 equation 0}, with $C=4\kappa$.
\end{proof}

Let $G_t(x):=\frac1{(2\pi t)^{\nicefrac{d}{2}}}\exp\left( -\frac{|x|^2}{2t}\right)$ be a Gaussian of variance $t$. Given a function $f$, we denote by $f_{*t}$ the convolution (in $x$) of $f$ with $G_t$. 

Lemma \ref{lemma 2} and Lemma \ref{lemma 3} are used to prove Corollary \ref{corollary 6}. \bbb
\begin{lemma}
\label{lemma 2}
For all $T>0$ we have
\begin{equation}
\label{lemma 2 equation 1}
\E \left[ \frac1T (\phi_T^\ext)^2+\frac1T |g_T|^2+|\nabla g_T|^2\right] \lesssim \frac1T\int_0^T  \left(\left|\E \left[|(\nabla\phi_T^\ext)_{*t}|^2\right]\right| + \E \left[ \left|(q_T-\E[q_T])_{*\frac{t}2}\right|^2\right]\right) \mathrm{d}t\,.
\end{equation}
\end{lemma}

\begin{proof}
By the property of convolution with a Gaussian, we have
\begin{equation}
\label{proof lemma 2 equation 0}
    \partial_t (\phi_T^\ext)_{*t}= \frac12 \Delta (\phi_T^\ext)_{*t}\,.
\end{equation}
The latter implies
\begin{equation}
\label{proof lemma 2 equation 1}
\partial_t [((\phi_T^\ext)_{*t})^2]=-| (\nabla\phi_T^\ext)_{*t}|^2+ \nabla \cdot ((\phi_T^\ext)_{*t}(\nabla\phi_T^\ext)_{*t})
\end{equation}
for all $t>0$. Here we simply multiplied \eqref{proof lemma 2 equation 0} by $2\,(\phi_T^\ext)_{*t}$ and used Leibniz's rule. By stationarity, from \eqref{proof lemma 2 equation 1} one obtains
\begin{equation}
\label{proof lemma 2 equation 2}
\frac{\mathrm{d}}{\mathrm{d}t} \E \left[ ((\phi_T^\ext)_{*t})^2\right]=-\E \left[| (\nabla\phi_T^\ext)_{*t}|^2\right].
\end{equation}
Now, any function $f:[a,b]\to \mathbb{R}$ continuous in $[a,b]$ and differentiable in $(a,b)$ satisfies
\begin{equation}
\label{proof lemma 2 equation 7}
\max_{t\in[a,b]}|f(t)|\le \int_a^b |f'(t)| \mathrm{d}t+ \frac{2}{b-a}\left|\int_{\frac{a+b}2}^b f(t)\,\mathrm{d}t \right|\,.
\end{equation}
Applying \eqref{proof lemma 2 equation 7} to the function $[0,T]\ni t\mapsto \E \left[ (\phi_T^\ext)_{*t}^2\right]$ one obtains
\begin{multline}
\label{proof lemma 2 equation 10}
\E \left[ ((\phi_T^\ext)_{*t})^2\right]
\le
\int_0^T \left|\frac{\mathrm{d}}{\mathrm{d}t} \E \left[ (\phi_T^\ext)_{*t}^2\right]\right|\mathrm{d}t
+
\frac2T \int_{\frac{T}2}^T  \E \left[ ((\phi_T^\ext)_{*t})^2\right]\mathrm{d}t
\\
\overset{\eqref{proof lemma 2 equation 2}}{\lesssim }
\int_0^T \E \left[| (\nabla\phi_T^\ext)_{*t}|^2\right]\mathrm{d}t +
 \frac2T \int_{\frac{T}2}^T  \E \left[ ((\phi_T^\ext)_{*t})^2\right]\mathrm{d}t\,.
\end{multline}

The task at hand then reduces to estimating the second integral on the RHS of \eqref{proof lemma 2 equation 10}. To this end, let us write $\phi_T^\ext$ as 
\begin{equation}
\label{proof lemma 2 equation 11}
\phi_T^\ext=\chi_\M\,\phi_T+(1-\chi_\M)\phi_T^\ext
\end{equation}
and, using the ``rob Peter to pay Paul'' inequality $(a+b)^2\le (1+\alpha^{-1})a^2+(1+\alpha)b^2$, decompose the integral in question accordingly:
\begin{multline}
\label{proof lemma 2 equation 12}
\frac2T \int_{\frac{T}2}^T  \E \left[ ((\phi_T^\ext)_{*t})^2\right]\mathrm{d}t
 \le
 \frac2T(1+\alpha^{-1})\int_{\frac{T}2}^T 
\E \left[ ((\chi_\M\,\phi_T)_{*t})^2\right]\mathrm{d}t
\\
+
\frac2T(1+\alpha)\int_{\frac{T}2}^T \E \left[ (((1-\chi_\M)\phi_T^\ext)_{*t})^2\right]\mathrm{d}t
\end{multline}
for all $\alpha>0$.

The second integral in the RHS of \eqref{proof lemma 2 equation 12} can be estimated via
Jensen's inequality in the form $(X_{*t})^2\le (X^2)_{*t}$ combined with stationarity:
\begin{multline}
\label{proof lemma 2 equation 13}
\frac2T(1+\alpha)\int_{\frac{T}2}^T \E \left[ (((1-\chi_\M)\phi_T^\ext)_{*t})^2\right]\mathrm{d}t
\\
\le
\frac2T(1+\alpha)\int_{\frac{T}2}^T \E \left[ ((1-\chi_\M)(\phi_T^\ext)^2)_{*t}\right] \mathrm{d}t
\\
\le
(1+\alpha)\E \left[ (1-\chi_\M)(\phi_T^\ext)^2\right]\,.
\end{multline}
As to the first integral in the RHS of \eqref{proof lemma 2 equation 12}, we argue as follows.
From \eqref{massive corrector} and \eqref{definition qT} we get
\begin{equation}
\label{proof lemma 2 equation 3}
\chi_\M \,\phi_T=T \,\nabla \cdot (q_T-\E[q_T])\,.
\end{equation}
With account of elementary properties of convolution with a Gaussian, \eqref{proof lemma 2 equation 3} implies
\begin{equation}
\label{proof lemma 2 equation 5}
(\chi_\M \phi_T)_{*t}
=
T((\nabla \cdot (q_T-\E[q_T]))_{*\frac{t}2})_{*\frac{t}2}
=
T(\nabla G_{\frac{t}2}* (q_T-\E[q_T])_{*\frac{t}2})
\end{equation}
which, by Jensen's inequality, gives us
\begin{multline}
\label{proof lemma 2 equation 6}
\E \left[ \left((\chi_\M\,\phi_T)_{*t}\right)^2\right]
\le
T^2\underset{\lesssim \frac1t}{\underbrace{\left(\int_{\R^d}\nabla G_{\frac{t}2}\right)^2}}\E \left[ \left|(q_T-\E[q_T])_{*\frac{t}2}\right|^2\right]
\\
\lesssim
\frac{T^2}t\,\E \left[ \left|(q_T-\E[q_T])_{*\frac{t}2}\right|^2\right]\,.
\end{multline}
Substituting \eqref{proof lemma 2 equation 13} and \eqref{proof lemma 2 equation 6} into \eqref{proof lemma 2 equation 12}, and, in turn, the latter into \eqref{proof lemma 2 equation 10} we obtain
\begin{multline}
\label{proof lemma 2 equation 14}
\E \left[ (\phi_T^\ext)^2\right]
\lesssim
\int_0^T \left|\E \left[| (\nabla\phi_T^\ext)_{*t}|^2\right]\right|\mathrm{d}t
\\
+
(1+\alpha^{-1})\int_{\frac{T}2}^T \E \left[ \left|(q_T-\E[q_T])_{*\frac{t}2}\right|^2\right]\mathrm{d}t
+
(1+\alpha)\E \left[ (1-\chi_\M)(\phi_T^\ext)^2\right]\,.
\end{multline}
Subtracting $\E \left[ (1-\chi_\M)(\phi_T^\ext)^2\right]$ from both sides of \eqref{proof lemma 2 equation 14}, we finally arrive at
\begin{multline}
\label{proof lemma 2 equation 15}
\E \left[\chi_\M (\phi_T^\ext)^2\right]
\lesssim
\int_0^T\E \left[| (\nabla\phi_T^\ext)_{*t}|^2\right]\mathrm{d}t
\\
+
(1+\alpha^{-1})\int_{\frac{T}2}^T \E \left[ \left|(q_T-\E[q_T])_{*\frac{t}2}\right|^2\right]
+
\alpha\,\E \left[ (1-\chi_\M)(\phi_T^\ext)^2\right]\,.
\end{multline}

By the property of the extension, there exists $c_0>0$ such that 
$c_0\,\E \left[(\phi_T^\ext)^2\right]\le\E \left[\chi_\M (\phi_T^\ext)^2\right]$. Hence, by choosing $\alpha$ appropriately with respect to $c_0$, \eqref{proof lemma 2 equation 15} gives us \eqref{lemma 2 equation 1}.
\end{proof}

\begin{lemma}
\label{lemma 3}
There exist constants $\epsilon=\epsilon(d,a_\pm)$ and $C=C(d,a_\pm)$ such that
\begin{equation}
\label{lemma 3 equation 1}
\frac{1}{R^d}\int_{\R^d} \left|\partial^\osc_{B_R(x)} (\nabla \phi_T^\ext,q_T)_{*t} \right|^2\,dx \le C  \left(\frac{\sqrt{T}}{\sqrt{t}} \right)^d \left(\min \left\{\frac{R}{\sqrt{T}},1\right\} \right)^{\epsilon d}
\end{equation}
for all $1\le t\le T$ and sufficiently big $R$.
\end{lemma}

\begin{proof}
Let $F(a):=(\nabla \phi_T^\ext,q_T)$. 

For all $x\in \R^d$ and $R>1$, definition \eqref{oscillation} implies
\begin{equation}
\label{proof lemma 3 equation 2}
\left|\partial^\osc_{B_R(x)} F(a) \right|\le \left|\partial^\osc_{\square_{2R}(x)} F(a) \right|\,,
\end{equation}
because in the RHS we are computing the oscillation over a larger set of admissible coefficients.
Let 
\begin{equation}
\label{proof lemma 3 equation 3}
\mathcal{P}:=\{z+\square_1 \ |\ z\in \Z^d\}
\end{equation}
be a partition of $\R^d$ into cubes of side $1$.

Then \eqref{proof lemma 3 equation 2} and purely algebraic arguments yield
\begin{multline}
\label{proof lemma 3 equation 4}
\int_{\R^d} \left|\partial^\osc_{B_R(x)} F(a) \right|^2 \,\mathrm{d}x \le \sum_{z\in 2R \Z^d} \int_{z+\square_{2R}}\left|\partial^\osc_{\square_{2R}(x)} F(a) \right|^2\,\mathrm{d}x
\\
=\int_{\square_{2R}} \sum_{Q\in x+2R\mathcal{P}} \left|\partial^\osc_{Q} F(a) \right|^2 \,\mathrm{d}x\,.
\end{multline}

Let $\{\zeta_Q\}_{Q\in R\mathcal{P}}$ be a real sequence and let $a_Q$ be coefficients that coincide with $a$ on shapes not fully contained in $Q$. We have
\begin{multline}
\label{proof lemma 3 equation 10}
\left(\sum_{Q\in R\mathcal{P}}\zeta_Q (F_{*t}(a_Q)-F_{*t}(a)) \right)^2
\overset{\text{Jensen}}{\le}
\left[\Bigl(\sum_{Q\in R\mathcal{P}}\zeta_Q (F(a_Q)-F(a)) \Bigr)^2\right]_{*t}
\\
\overset{G_t\lesssim (T/t)^{d/2}\omega_T}{\lesssim}
\left(\frac{T}{t} \right)^\frac{d}2 \int \omega_T \left(\sum_{Q\in R\mathcal{P}}\zeta_Q (F(a_Q)-F(a)) \right)^2.
\end{multline}
Now, for a compactly supported test function $v$ and each cube $Q$, we have
\begin{multline}
\label{proof lemma 3 equation 100}
\frac1T\int_{Q} \chi_{\M} \left(\phi_T^{\ext}-\phi_{T,Q}^{\ext,Q}\right) v+\int_{Q}a\left(\nabla\phi_T^{\ext}-\nabla\phi_{T,Q}^{\ext,Q}\right)\nabla v
\\
=
-\frac1T\int_{Q} (\chi_{\M}-\chi_{\M_Q}) \phi_{T,Q}^{\ext,Q} v-\int_{Q}(a-a_Q)\left( \nabla\phi_{T,Q}^{\ext,Q}+\mathbf{e}\right)\nabla v\,.
\end{multline}
Let us decompose $\nabla\phi_T^{\ext}-\nabla\phi_{T,Q}^{\ext,Q}$ as
\begin{equation}
\label{proof lemma 3 equation 101}
  \nabla\phi_T^{\ext}-\nabla\phi_{T,Q}^{\ext,Q}=\nabla\phi_T^{\ext}-(\nabla\phi_{T,Q}^{\ext,Q})^\ext+(\nabla\phi_{T,Q}^{\ext,Q})^\ext-\nabla\phi_{T,Q}^{\ext,Q}
\end{equation}
and deal with the two terms in the RHS separately.

Multiplying \eqref{proof lemma 3 equation 100} by $\zeta_Q$ and summing over all cubes we obtain from Lemma~\ref{prop:omega_T estimate}
\begin{multline}
\label{proof lemma 3 equation 102}
     \int \omega_T\left|\sum_{Q\in R\mathcal{P}}\zeta_Q \left(\nabla\phi_T^{\ext}-(\nabla\phi_{T,Q}^{\ext,Q})^\ext\right) \right|^2
     \\
     \lesssim
  \left(\sum_{Q\in R\mathcal{P}}\zeta_Q^2\right) \sup_{Q'\in R\mathcal{P}} \int_{Q'} \omega_T\left( \frac{1}{T}(\phi_{T,{Q'}}^{\ext,{Q'}})^2+ |\nabla\phi_{T,{Q'}}^{\ext,{Q'}}+\mathbf{e}|^2\right).
\end{multline}
As to the second contribution to \eqref{proof lemma 3 equation 101}, it can be estimated directly using the extension property as follows:
\begin{equation}
\label{proof lemma 3 equation 103}
     \int \omega_T\left|\sum_{Q\in R\mathcal{P}}\zeta_Q \left((\nabla\phi_{T,Q}^{\ext,Q})^\ext-\nabla\phi_{T,Q}^{\ext,Q}\right) \right|^2
     \lesssim
  \left(\sum_{Q\in R\mathcal{P}}\zeta_Q^2\right) \sup_{Q'\in R\mathcal{P}} \int_{Q'} \omega_T  |\nabla\phi_{T,{Q'}}^{\ext,{Q'}}+\mathbf{e}|^2.
\end{equation}
Lemma~\ref{lemma hole filling} combined with \eqref{29 January 2024 equation 2} then implies
\begin{equation}
\label{proof lemma 3 equation 19}
\int_Q \left(\frac1T  (\phi_{T,Q}^{\ext,Q})^2
+
\left|\nabla\phi_{T,Q}^{\ext,Q}+\mathbf{e} \right|^2\right)\lesssim \left( \frac{R}{\sqrt{T}} \right)^{\epsilon d} \sqrt{T}^d
\end{equation}
which, in turn, gives us
\begin{equation}
\label{proof lemma 3 equation 20}
\int_Q \omega_T \left(\frac1T  (\phi_{T,Q}^{\ext,Q})^2
+
\left|\nabla\phi_{T,Q}^{\ext,Q}+\mathbf{e} \right|^2\right)\lesssim \left( \frac{R}{\sqrt{T}} \right)^{\epsilon d}
\end{equation}
by definition of $\omega_T$ (see~\eqref{exponential weight}). 

Next, we deal with the component of $F(a)$ involving the flux. 
Since, in view of \eqref{definition qT}, 
\begin{equation}
\label{q_{T,Q}-q_T}
    q_{T,Q}-q_T=(a_Q-a)(\nabla\phi_{T,Q}^{\ext,Q}+\mathbf{e})+a(\nabla \phi_{T,Q}^{\ext,Q}-\nabla \phi_T^{\ext})\,,
\end{equation}
using \eqref{proof lemma 3 equation 102} and \eqref{proof lemma 3 equation 103} we immediately get
\begin{multline}
    \label{proof lemma 3 equation 104}
     \int \omega_T\left|\sum_{Q\in R\mathcal{P}}\zeta_Q \left(q_{T,Q}-q_T\right) \right|^2
     \\
     \lesssim
  \left(\sum_{Q\in R\mathcal{P}}\zeta_Q^2\right) \sup_{Q'\in R\mathcal{P}} \int_{Q'} \omega_T\left( \frac{1}{T}(\phi_{T,{Q'}}^{\ext,{Q'}})^2+ |\nabla\phi_{T,{Q'}}^{\ext,{Q'}}+\mathbf{e}|^2\right).
\end{multline}
Thus, substituting \eqref{proof lemma 3 equation 101} and \eqref{q_{T,Q}-q_T} into \eqref{proof lemma 3 equation 10}, and using \eqref{proof lemma 3 equation 102}, \eqref{proof lemma 3 equation 103}, \eqref{proof lemma 3 equation 104} and~\eqref{proof lemma 3 equation 20}, we arrive at
\begin{multline}
\label{proof lemma 3 equation 21}
\left(\sum_{Q\in R\mathcal{P}}\zeta_Q (F_{*t}(a_Q)-F_{*t}(a)) \right)^2
\\
 \lesssim
 \left(\frac{T}{t} \right)^\frac{d}2 
\left(\sum_{Q\in R\mathcal{P}}
\zeta_Q^2\right) \sup_{Q'\in R\mathcal{P}} \int_{Q'} \omega_T \left(\frac1T  (\phi_{T,{Q'}}^{\ext,{Q'}})^2
+
\left|\nabla\phi_{T,{Q'}}^{\ext,{Q'}}+\mathbf{e} \right|^2\right)
\\
\lesssim
 \left(\frac{\sqrt{T}}{\sqrt{t}} \right)^d \left(\min \left\{\frac{R}{\sqrt{T}},1\right\} \right)^{\epsilon d}\sum_{Q\in R\mathcal{P}}
\zeta_Q^2\,.
\end{multline}

By discrete duality, combining formulae \eqref{proof lemma 3 equation 4} and~\eqref{proof lemma 3 equation 21} one obtains~\eqref{lemma 3 equation 1}.
\end{proof}
	\begin{corollary}
		\label{corollary 6}
		There exists $C=C(d,a_\pm,T_0)$ and $\epsilon=\epsilon(d,a_\pm)$  such that for all $T\geq T_0>1$ we have
		\begin{equation}
			\label{corollary 6 equation 1}
			\E \left[\frac1T (\phi_T^\ext)^2+\frac1T |g_T|^2+|\nabla g_T|^2\right]\le C T^{-\epsilon}\,.
		\end{equation}
	\end{corollary}
	
	\begin{proof}
	
		By Lemma~\ref{lemma 2}, it is enough to estimate the RHS of~\eqref{lemma 2 equation 1}. To this end, we begin by observing that Jensen's inequality implies $|(\nabla\phi_T^\ext)_{*t}|^2\le |\nabla\phi_T^\ext|^2_{*t}$. 
		Furthermore, we have 
		$$ \left|(q_T-\E[q_T])_{*\frac{t}2}\right|^2 \leq 
		\left|q_T-\E[q_T]\right|^2_{*\frac{t}{2}} \lesssim |\nabla\phi_T^\ext|^2_{*\frac{t}{2}}+\E\left[|\nabla\phi_T^\ext|^2\right]. $$
		Therefore, by stationarity we get
		\begin{multline}
			\label{proof corollary 6 equation 1}
			\frac1T\int_0^{T^{1-\epsilon}}  \left(\E \left[|(\nabla\phi_T^\ext)_{*t}|^2\right] + \E \left[ \left|(q_T-\E[q_T])_{*\frac{t}2}\right|^2\right]\right) \dr t
			\lesssim
			T^{-\epsilon}
			\E\left[|\nabla \phi_T^\ext|^2\right]
			\lesssim 
			T^{-\epsilon}\,.
		\end{multline}
		
		Using the Spectral Gap Inequality~\eqref{spectral gap}, we can control $\E \left[|(\nabla\phi_T^\ext)_{*t}|^2\right]$ by means of its oscillation:
		\begin{equation}
			\label{proof corollary 6 equation 2}
			\E \left[|(\nabla\phi_T^\ext)_{*t}|^2\right]\le \int_{\R^d}\left|\partial^\osc_{B_1(x)} (\nabla \phi_T^\ext)_{*t} \right|^2\,\dr x.
		\end{equation}
		Hence, for $d>2$  Lemma~\ref{lemma 3} implies 
		\begin{multline}
			\label{proof corollary 6 equation 3}
			\frac1T\int_{T^{1-\epsilon}}^T \E \left[|(\nabla\phi_T^\ext)_{*t}|^2\right] \dr t 
			\lesssim
			\frac1T\int_{T^{1-\epsilon}}^T \int_{\R^d}\left|\partial^\osc_{B_1(x)} (\nabla \phi_T^\ext)_{*t}\right|^2\,\dr x
			\\
			\lesssim
			\frac1T\int_{T^{1-\epsilon}}^T 
			\left(\frac{\sqrt{T}}{\sqrt{t}} \right)^d \sqrt{T}^{-\epsilon d} \,\dr t
			\\
			\lesssim  T^{-1+\frac{d}2(1-\epsilon)}
			\int_{T^{1-\epsilon}}^T t^{-\frac{d}2} \,\dr t
			\\
			\lesssim
			T^{-1+\frac{d}2(1-\epsilon)}
			(T^{(-\frac{d}2+1)(1-\epsilon)}-T^{-\frac{d}2+1})
			\\
			=
			T^{-\epsilon}-T^{-\epsilon\frac{d}2}\,.
		\end{multline} 
	For $d=2$ one decreases the value $\epsilon$ to re-absorb the logarithm arising from the integration.  
		
		Combining \eqref{proof corollary 6 equation 2} and \eqref{proof corollary 6 equation 3} with Lemma~\ref{lemma 2} we arrive at \eqref{corollary 6 equation 1}.
\end{proof}

\begin{lemma}[Concentration for averages {\cite[Proposition~4.3]{DG1}}]\label{lemma 4}
    Suppose Assumption~\ref{assumption multiscale spectral gap} is satisfied.   
    Let $t\ge 1$ and let $F_t$ denote a bounded random variable
		that is approximately $\sqrt{t}$-local in the sense of \eqref{31 January 2024 prop equation 2}\footnote{Note that $t$ here plays the role of what was denoted by $T$ in Proposition~\ref{prop:F_t local}.}.
		Then there exists a positive constant $C=C(d,a_\pm,\pi)$ such that for 
		all $\delta>0$ and $T\geq t$ one has
		\begin{equation}\label{lemma 4 equation 1}
			\Pro\left(\left|\int \omega_T(x)F_t(a(\cdot+x))\,dx-\E (F_t)\right|>\delta\right)\, \le \, \exp\left(-\frac{\min\{\delta,\delta^2\}} {C} \big(\frac{\sqrt{T}}{\sqrt{t}}\big)^{\min\{d/2,\beta\}}\right).
		\end{equation}
If Assumption~\ref{assumption spectral gap} is satisfied, then one can take $\beta=\infty$ in \eqref{lemma 4 equation 1}.
	\end{lemma}
\igor
Finally, we are in a position to prove our main result.

\begin{proof}[Proof of Theorem~\ref{theorem 2 equation 1}]
		The claim can be proved by retracing, modulo minor natural adaptations, the proof of~\cite[Theorem 2]{mjm} using all the lemmata and propositions above for perforated domains, which replace in the arguments those obtained in \cite{mjm} for uniformly elliptic problems. We will explain the idea of the proof. By using Proposition \ref{proposition 3} one can control $r_{**}$ with the right hand side of \eqref{propositon 3 equation 1}. On the other hand the right hand side of \eqref{propositon 3 equation 1} has exponentially decaying tails as a consequence of Lemma  \ref{lemma 4} after we subtract $\E[F_t]$.  On the other hand, by using Corollary \ref{corollary 6} we see that $\E[F_t]$ can be neglected for $t$ large enough $1<t\leq T=R^2$. 
    \end{proof}
    
The next proposition supplies a sensitivity estimate that is a crucial ingredient in the proof of Theorem~\ref{bella kniely theorem 1.13}. The proof of the upcoming proposition is an appropriate adaptation to our setting of \cite[Theorem 1.13]{bella_kniely}. Particular care is needed due to the fact that we need to work with extension of the corrector. 
\bbb 
\begin{proposition}
\label{bella kniely proposition 1.12}
\label{proposition 1.12}
Let
\begin{equation}
\label{proposition 1.12 equation 1}
F\nabla \phi=\int g \cdot \nabla\phi^\ext\,,
\end{equation}
where $g:\mathbb{R}^d\to \mathbb{R}^d$ is a smooth averaging function compactly supported in $B_r$, $\left( \int |g|^2\right)^{1/2}\lesssim  r^{-d/2}$.
Then we have
\begin{equation}
\label{proposition 1.12 equation 2}
\int_{\R^d} \left|\partial^\osc_{B_M(x)} F\nabla\phi(a) \right|^2 \,\dr x\le C \left(\frac{(r+r_*)^{1-\epsilon}}{r}\right)^d\,.
\end{equation}
\end{proposition}

\begin{proof}
The proof retraces that of \cite[Proposition~1.12]{bella_kniely}. In what follows we will adopt the structure of the latter proof, elaborating in detail on the modifications required and sketchily recalling the rest. For the sake of simplicity, we will put $M=1$.

Let us first suppose that $r>r_*$.

\

Step 1 of \cite[Proof of Proposition~1.12]{bella_kniely} is essentially hole filling, which we have proved in our setting in Lemma~\ref{lemma hole filling} for the massive corrector $\phi_T$. The argument for the corrector $\phi$ is analogous and gives us
\begin{equation}
\label{proof of proposition 1.12 equation 1}
\int_{B_r(x)} |\nabla \phi^\ext+\mathbf{e}|^2\lesssim \left(\frac{r}{R}\right)^{\epsilon d}\int_{B_R(x)} |\nabla \phi^\ext+\mathbf{e}|^2.
\end{equation}

With regards to Step 2, the estimate
\begin{equation}
\label{proof of proposition 1.12 equation 2}
\fint_{B_R}|\nabla\phi^\ext+\mathbf{e}|^2 \lesssim 1 \qquad \text{for all }R\ge r_*+1
\end{equation}
follows from the sublinearity of the corrector, see Proposition~\ref{theorem 1}. Furthermore, by closely adapting the arguments from \cite[Proof of Corollary 3]{GNOV2}, it is not hard to show that, for $r\ge r_*$, we have
\begin{equation}
\label{proof of proposition 1.12 equation 3}
\int_{B_r}|\nabla u^\ext|^2\lesssim \int \left(\frac{|x|}{r}+1\right)^{-\gamma} |g|^2 \qquad \text{for all\ } 0<\gamma<d\,,
\end{equation}
and its generalisation
\begin{equation}
\label{proof of proposition 1.12 equation 3 generalisation}
\int\left(\frac{|x|}{r}+1\right)^{-\gamma} |\nabla u^\ext|^2\lesssim \int \left(\frac{|x|}{r}+1\right)^{-\gamma'} |g|^2 \qquad \text{for all\ } 0<\gamma'<\gamma<d\,,
\end{equation}
where $u$ is related to $g$ via the equation $-\nabla \cdot a \nabla u=\nabla\cdot (\chi_\M g)$ understood in a weak sense.

Let us then move to Step 3.
Let us observe that
\begin{enumerate}[(i)]
\item
$\left|\frac{\partial F\nabla\phi}{\partial^\mathrm{osc}_y a}(a) \right|^2 \lesssim 2 \sup_{a_y}|F\nabla\phi-F\nabla\phi_y|^2$;

\item
the oscillation over a set is bounded above by the oscillation over any larger set that contains it.
\end{enumerate}
Then
\begin{equation}
\label{proof of proposition 1.12 equation 3.1}
\int_{\R^d} \left|\frac{\partial F\nabla\phi}{\partial^\mathrm{osc}_x a}(a) \right|^2 \,\dr x
\lesssim
\sum_{y\in \Z^d}\left|\frac{\partial F\nabla\phi}{\partial^\mathrm{osc'}_y a}(a) \right|^2.
\end{equation}
where $\osc'$ indicates that we are taking the oscillations over (larger) balls of radius $3$. Hence, it suffices to estimate the RHS of \eqref{proof of proposition 1.12 equation 3.1}.

Let us denote by $a_y$ coefficients that may differ from $a$ in a ball of radius $3$ centred at $y\in \R^d$, and denote by $\phi_y$ the corrector corresponding to $a_y$. Then \eqref{corrector equation} implies
\begin{equation}
\label{proof of proposition 1.12 equation 4}
-\nabla \cdot a \nabla(\phi-\phi_y^{\ext,y})=\nabla \cdot (a-a_y)(\nabla \phi_y^{\ext,y}+\mathbf{e})\,,
\end{equation}
\begin{equation}
\label{proof of proposition 1.12 equation 5}
-\nabla \cdot a_y \nabla(\phi_y-\phi^\ext)=\nabla \cdot (a_y-a)(\nabla \phi^{\ext}+\mathbf{e})\,.
\end{equation}
Here $(\,\cdot\,)^{\ext,y}$ denotes the extension inside the perforation generated by $a_y$, where $(\,\cdot\,)^{\ext}$ is the extension with respect to the perforation generated by the original coefficient $a$.
The RHS's of equations \eqref{proof of proposition 1.12 equation 4} and \eqref{proof of proposition 1.12 equation 5} are supported in $B_{3}(y)$; thus, by \eqref{proof of proposition 1.12 equation 3} we get
\begin{equation}
\label{proof of proposition 1.12 equation 6}
\int_{B_r} |\nabla[(\phi-\phi_y^{\ext,y})]^{\ext}|^2 \lesssim \sup_{x\in B_{3}(y)}|a-a_y|\int_{B_{3}(y)} \left(\frac{|x|}{r} +1\right)^{-\gamma}|\nabla \phi_y^{\ext,y}+\mathbf{e}|^2\,\dr x\,,
\end{equation}
\begin{equation}
\label{proof of proposition 1.12 equation 7}
\int_{\R^d} |\nabla[(\phi_y-\phi^{\ext})]^{\ext,y}|^2 \lesssim \sup_{x\in B_{3}(y)}|a-a_y|\int_{B_{3}(y)} |\nabla \phi^{\ext}+\mathbf{e}|^2\,\dr x\,,
\end{equation}
and by \eqref{proof of proposition 1.12 equation 3 generalisation} we get
\begin{equation}
\label{proof of proposition 1.12 equation 6bis}
\int \left(\frac{|x|}{r} +1\right)^{-\gamma}|\nabla[(\phi-\phi_y^{\ext,y})]^{\ext}|^2 \lesssim \sup_{x\in B_{3}(y)}|a-a_y|\int_{B_{3}(y)} 
\left(\frac{|x|}{r} +1\right)^{-\gamma'}|\nabla \phi_y^{\ext,y}+\mathbf{e}|^2\,\dr x\,.
\end{equation}

Formula \eqref{proof of proposition 1.12 equation 4} implies that for any $\{c_y\}_{y\in\Z^d}$ we have
\begin{equation}
\label{temp3}
-\nabla \cdot a \nabla\sum_{y\in\Z^d}c_y(\phi-\phi_y^{\ext,y})=\nabla \cdot \sum_{y\in\Z^d}c_y (a-a_y)(\nabla \phi_y^{\ext,y}+\mathbf{e})\,.
\end{equation}
Then, using \eqref{proof of proposition 1.12 equation 3} and \eqref{proof of proposition 1.12 equation 3 generalisation}, from \eqref{temp3} we get
\begin{multline}
\label{temp2}
    \int_{B_r} \left|\nabla \sum_{y\in\Z^d}c_y[\phi-\phi_y^{\ext,y}]^\ext \right|^2+ \int \left(\frac{|x|}{r} +1\right)^{-\gamma}\left|\nabla \sum_{y\in\Z^d}c_y[\phi-\phi_y^{\ext,y}]^\ext \right|^2
    \\
    \lesssim 
    \int \left(\frac{|x|}{r} +1\right)^{-\gamma'}\left|
    \sum_{y\in\Z^d}c_y (a-a_y)(\nabla \phi_y^{\ext,y}+\mathbf{e})
    \right|^2
    \\
    \lesssim
    \sum_{y\in\Z^d}c_y^2 \sup_{x\in B_3(y)}|a-a_y|^2 \int_{B_{3}(y)} \left(\frac{|x|}{r} +1\right)^{-\gamma'}|\nabla \phi_y^{\ext,y}+\mathbf{e}|^2\,.
\end{multline}

\

Straightforward algebraic manipulations yield
\begin{equation}
  \label{temp1}
  a(\nabla\phi+\mathbf{e})-a_y(\nabla\phi_y^{\ext,y}+\mathbf{e})=a\nabla(\phi-\phi_y^{\ext,y})+(a-a_y)(\nabla\phi_y^{\ext,y}+\mathbf{e}),
\end{equation}
which, combined with \eqref{temp2}, leads to
\begin{multline}
\label{temp5}
   \int_{B_r} \left| \sum_{y\in\Z^d}c_y\nabla[\phi-\phi_y^{\ext,y}]^\ext \right|^2+ \int \left(\frac{|x|}{r} +1\right)^{-\gamma}\left| \sum_{y\in\Z^d}c_y[a(\nabla\phi+\mathbf{e})-a_y(\nabla\phi_y^{\ext,y}+\mathbf{e})] \right|^2
      \\
    \overset{\eqref{temp2},\, \gamma'<\gamma}{\lesssim}
    \sum_{y\in\Z^d}c_y^2 \sup_{x\in B_3(y)}|a-a_y|^2 \int_{B_3(y)} \left(\frac{|x|}{r} +1\right)^{-\gamma'}|\nabla \phi_y^{\ext,y}+e|^2\,.
\end{multline}

Now, by definition of $F$ we have
\begin{equation}
\label{proof of proposition 1.12 equation 8}
|F\psi|=\left|\int g\cdot \psi \right|\le \left( \int_{\operatorname{supp} g} |\psi|^2 \right)^{1/2} \left( \int |g|^2\right)^{1/2} 
\lesssim 
\left( \fint_{B_r} |\psi|^2 \right)^{1/2}
\end{equation}
where the last estimate is a consequence of the fact that $g$ is an averaging function, i.e., $\left( \int |g|^2\right)^{1/2}\lesssim r^{-d/2}$.


Then for any $\{c_y\}_{y\in\Z^d}$ we have
\begin{multline}
\label{proof of proposition 1.12 equation 9new}
\left|\sum_{y\in \Z^d}c_y (F\nabla\phi^\ext-F\nabla\phi_y^{\ext,y})\right|^2=\left|\sum_{y\in \Z^d}c_y [F(\nabla\phi^\ext-(\nabla\phi_y^{\ext,y})^\ext)+F(\nabla\phi_y^{\ext,y})^\ext-F\nabla\phi_y^{\ext,y})]\right|^2
\\
\lesssim
\left|\sum_{y\in \Z^d}c_y F(\nabla\phi^\ext-(\nabla\phi_y^{\ext,y})^\ext)\right|^2
+
\left|\sum_{y\in \Z^d}c_y [F((\nabla\phi_y^{\ext,y})^\ext+\mathbf{e})-F(\nabla\phi_y^{\ext,y}+\mathbf{e})]\right|^2\,.
\end{multline}
The first sum in the RHS of \eqref{proof of proposition 1.12 equation 9new} can be estimated using \eqref{proof of proposition 1.12 equation 8} to obtain
\begin{multline}
\left|\sum_{y\in \Z^d}c_y F(\nabla\phi^\ext-(\nabla\phi_y^{\ext,y})^\ext)\right|^2
\lesssim
\fint_{B_r} \left| \sum_{y\in\Z^d}c_y\nabla[\phi-\phi_y^{\ext,y}]^\ext \right|^2
\\
\overset{\eqref{temp5}}{\lesssim}
  r^{-d}\sum_{y\in\Z^d}c_y^2 \sup_{x\in B_3(y)}|a-a_y|^2 \int_{B_3(y)} \left(\frac{|x|}{r} +1\right)^{-\gamma'}|\nabla \phi_y^{\ext,y}+\mathbf{e}|^2\,.
\end{multline}
As to the second sum in the RHS of \eqref{proof of proposition 1.12 equation 9new}, we have
\begin{multline}
    \label{temp6}
    \left|\sum_{y\in \Z^d}c_y [F((\nabla\phi_y^{\ext,y})^\ext+\mathbf{e})-F(\nabla\phi_y^{\ext,y}+\mathbf{e})]\right|^2
    \\
    =
     \left|\sum_{y\in \Z^d}c_y [F(\chi_{B_3(y)}((\nabla\phi_y^{\ext,y})^\ext+\mathbf{e}))-F(\chi_{B_3(y)}(\nabla\phi_y^{\ext,y}+\mathbf{e}))]\right|^2
    \\
    \lesssim
     r^{-d}  \sum_{y\in\Z^d}c_y^2 \sup_{x\in B_3(y)}|a-a_y|^2 \left(
       \int_{B_r}\chi_{B_3(y)}|(\nabla\phi_y^{\ext,y})^\ext+\mathbf{e}|^2
+
\int_{B_r}\chi_{B_3(y)}|\nabla\phi_y^{\ext,y}+\mathbf{e}|^2
       \right)
       \\
       \lesssim
     r^{-d}  \sum_{y\in\Z^d}c_y^2 \sup_{x\in B_3(y)}|a-a_y|^2 
\int_{B_{r+1}}\chi_{B_{4}(y)}|\nabla\phi_y^{\ext,y}+\mathbf{e}|^2
        \\
       \lesssim
     r^{-d}  \sum_{y\in\Z^d}c_y^2 \sup_{x\in B_3(y)}|a-a_y|^2 
\int_{B_{4}(y)}\left(\frac{|x|}{r} +1\right)^{-\gamma'}|\nabla\phi_y^{\ext,y}+\mathbf{e}|^2
      \,,
\end{multline}
where in the last step we used that $1\lesssim \left(\frac{|x|}{r} +1\right)$ in $B_{r+1}$.

By an $\ell^2$ duality argument, we conclude from \eqref{proof of proposition 1.12 equation 9new}--\eqref{temp6} that
\begin{multline}
    \label{temp7}
   r^d\sum_{y\in\Z^d} |F\nabla\phi^\ext-F\nabla\phi_y^{\ext,y}|^2 \\\lesssim 
  \sup_{y\in\Z^d}  \left(\sup_{x\in B_3(y)}|a-a_y|^2
\int_{B_4(y)}\left(\frac{|x|}{r} +1\right)^{-\gamma'}|\nabla\phi_y^{\ext,y}+\mathbf{e}|^2
       \right)
       \\
        \lesssim
      \sup_{y\in\Z^d}  \left(\sup_{x\in B_3(y)}|a-a_y|^2\left(\frac{|y|}{r} +1\right)^{-\gamma'}
\int_{B_4(y)} |\nabla\phi_y^{\ext,y}+\mathbf{e}|^2
       \right).
\end{multline}
Resorting to the inequality \eqref{proof of proposition 1.12 equation 7}, we obtain from \eqref{temp7} that
\begin{multline}
\label{temp9}
   r^{d} \int_{\R^d} \left|\frac{\partial F\nabla\phi}{\partial^\mathrm{osc}_x a}(a) \right|^2 \,\dr x
       \lesssim
      \sup_{y\in\Z^d}  \left(\sup_{x\in B_3(y)}|a-a_y|^2\left(\frac{|y|}{r} +1\right)^{-\gamma'}
\int_{B_4(y)} |\nabla\phi^{\ext}+\mathbf{e}|^2
       \right).
\end{multline}
We should emphasise that obtaining the above estimate via \eqref{proof of proposition 1.12 equation 7} and triangular inequality involves the use of the fact that $(\,(\,\cdot\,)^{\ext}\,)^{\ext,y}$ is dominated by $(\,\cdot\,)^{\ext}$.

Putting $R:=|y|$, from \eqref{proof of proposition 1.12 equation 1} we get 
\begin{equation}
\label{proof of proposition 1.12 equation 10}
\int_{B_{3}(y)} |\nabla \phi^{\ext}+\mathbf{e}|^2\lesssim \left( \frac{3}{R+3}\right)^{\epsilon d}\int_{B_{R+3}(y)} |\nabla \phi^{\ext}+\mathbf{e}|^2\,.
\end{equation}


Now, suppose $r_*\le 2R+3$. Then 
\eqref{proof of proposition 1.12 equation 10} implies
\begin{equation}
\label{proof of proposition 1.12 equation 12}
\int_{B_{3}(y)} |\nabla \phi^{\ext}+\mathbf{e}|^2\lesssim \left( \frac{1}{R+r_*+1}\right)^{\epsilon d}\int_{B_{2R+3}} |\nabla \phi^{\ext}+\mathbf{e}|^2\,.
\end{equation}

If, on the other hand, $r_*>2R+3$, then hole-filling gives us a factor
\begin{equation}
\label{proof of proposition 1.12 equation 13}
\left( \frac{1}{R+3}\right)^{\epsilon d}
\left(
\frac{R+3}{r_*}
\right)^{\epsilon d}
=
\left(
\frac{1}{r_*}
\right)^{\epsilon d}
\lesssim
\left(
\frac{1}{R+r_*+1}
\right)^{\epsilon d}\,,
\end{equation}
thus getting once again \eqref{proof of proposition 1.12 equation 12} with the integral in the RHS taken over $B_{r_*}$.
(Recall that $r_*\ge 1$.)

All in all, by \eqref{proof of proposition 1.12 equation 2} and 
\eqref{proof of proposition 1.12 equation 12} we arrive at
\begin{multline}
\label{proof of proposition 1.12 equation 14}
\int_{B_{3}(y)} |\nabla \phi^{\ext}+\mathbf{e}|^2\lesssim \left( \frac{1}{R+r_*+1}\right)^{\epsilon d}(R+r_*+1)^d
\\
=
(R+r_*+1)^{d(1-\epsilon)}\le (R+r+1)^{d(1-\epsilon)}\,,
\end{multline}
because we assumed $r>r_*$. Recalling \eqref{temp9} we obtain the claim \eqref{proposition 1.12 equation 2} as follows. If $R<r$, then since $|y|=R$ we have $\left(\frac{|y|}{r} +1\right)^{-\gamma'}<1$ hence we conclude. If, on the other hand, $R>r$, then we estimate $\left(\frac{|y|}{r} +1\right)^{-\gamma'}<\left(\frac{R}{r}
\right)^{-\gamma'}$ and by choosing $\gamma'=d(1-\epsilon)$ we get
\[
\left(\frac{R}{r}
\right)^{-\gamma'}(R+r+1)^{d(1-\epsilon)}\lesssim r^{d(1-\epsilon)}\,,
\]
and we also conclude.

Lastly, we examine the situation when $r<r_*$. Following \cite{bella_kniely}, consider the rescaled functional $\tilde F\psi:
=\left(\frac{r}{r_*}\right)^\frac{d}2 F\psi$.
By \eqref{proof of proposition 1.12 equation 8} we have
\begin{equation}
|\tilde F\psi|^2=\left(\frac{r}{r_*}\right)^d |F\psi|^2
\lesssim
\left(\frac{r}{r_*}\right)^d
\fint_{B_r} |\psi|^2
\lesssim 
\fint_{B_{r_*}} |\psi|^2\,.
\end{equation}
Hence, by applying the argument from the previous steps to $\tilde F$ with $r$ replaced by $r_*$ one obtains the claim.
\end{proof}

\section*{Acknowledgements}
\addcontentsline{toc}{section}{Acknowledgements}
Matteo Capoferri and Mikhail Cherdantsev are grateful for the hospitality to TU Dortmund and the University of Split, where part of this work has been done.

\

\emph{Funding:} Peter Bella was supported by German Science Foundation DFG under the Grant Agreement Nr. 441469601. Matteo Capoferri was supported by EPSRC Fellowship EP/X01021X/1. Mikhail Cherdantsev was supported by the Leverhulme Trust Research Project Grant RPG-2019-240. Igor Vel\v{c}i\'c was supported by the Croatian Science Foundation under the project number~HRZZ-IP-2022-10-5181. 

\begin{appendices}

\section{Hole filling lemma in perforated domains}
\label{Appendix hole filling}

    The purpose of this appendix is to provide a proof of a version of the classical hole filling lemma in the setting of perforated domains.
    
    For this, we shall need the following simple lemma.

\begin{lemma}
\label{lemma about contradiction argument}
For $R\ge1$ let $\eta_R$ be a smooth cut-off of $B_{R/2}$ in $B_R$. Put 
\begin{equation}
\label{lemma about contradiction argument equation 1}
F_{\omega,R}(u):=\frac{\int_{B_R}\eta_R^2 \chi_\M u}{\int_{B_R}\eta_R^2 \chi_\M}\,, \qquad u\in L^2(B_R)\,.
\end{equation}
Then we have
\begin{equation}
\label{lemma about contradiction argument equation 2}
\int_{B_R} \chi_\M |u-F_{\omega,R}(u)|^2\le C \inf_c\int_{B_R} |u^\ext-c|^2\,,
\end{equation}
    where $C$ is a positive constant independent of $\omega$ and $R$.
\end{lemma}

\begin{proof}
Let us begin by observing that the functional $F_{\omega,R}$ is linear and that constants are fixed points for $F_{\omega,R}$, namely
\begin{equation}
\label{proof lemma about contradiction argument equation 1}
F_{\omega,R}(c)=c \qquad \forall c \in \mathbb{R}.
\end{equation}
Arguing by contradiction and rescaling $B_R$ to $B_1$, suppose that for every $n\in \mathbb{N}$ there exists $u_n\in L^2(B_1)$, $u_n\neq0$, and $\omega_n\in R^{-1}\Omega$ such that
\begin{equation}
\label{proof lemma about contradiction argument equation 2}
\int_{B_1} \chi_{\M_n} |u_n-F_{\omega_n,1}(u_n)|^2> n \inf_c\int_{B_1} |u_n^\ext-c|^2\,.
\end{equation}
Without loss of generality, we can assume that 
\begin{equation}
\label{proof lemma about contradiction argument equation 3}
\int_{B_1} \chi_{\M_n} |u_n-F_{\omega_n,1}(u_n)|^2=1
\end{equation}
and that $\fint_{B_1} u_n^\ext=0$. The former can be achieved by rescaling $u_n$, the latter by subtracting a constant from $u_n$ and using \eqref{proof lemma about contradiction argument equation 1}. Thus, we can rewrite \eqref{proof lemma about contradiction argument equation 2} as
\begin{equation}
\label{proof lemma about contradiction argument equation 4}
\int_{B_1} |u_n^\ext|^2< \frac{1}{n}, 
\end{equation}
which implies that
\begin{equation}
\label{proof lemma about contradiction argument equation 5}
u_n^\ext \to 0
\end{equation}
strongly in $L^2(B_1)$. 
But $\int_{B_1}\eta_1^2 \chi_{\M_n}$ is uniformly bounded from below and $\|\chi_{\M_n
}\|_{L^\infty}=1$, hence \eqref{proof lemma about contradiction argument equation 5} implies 
\begin{equation}
\label{proof lemma about contradiction argument equation 6}
F_{\omega_n,1}(u_n^\ext)=F_{\omega_n}(u_n)\to 0 \qquad\text{as}\quad n \to +\infty.
\end{equation}

Formulae \eqref{proof lemma about contradiction argument equation 5} and \eqref{proof lemma about contradiction argument equation 6} contradict \eqref{proof lemma about contradiction argument equation 3}.
\end{proof}

\begin{remark}
The same argument allows one to prove
\begin{equation}
\label{lemma about contradiction argument equation 2 bis}
\int_{B_R} \chi_\M |u-F_{\omega,R}(u)|^2\le C \inf_c\int_{B_R} \chi_\M |u-c|^2\,,
\end{equation}
instead of \eqref{lemma about contradiction argument equation 2} --- note the difference in the RHS.
\end{remark}

\begin{lemma}[Hole filling]
\label{lemma hole filling}
For $1<R\le \sqrt{T}$ we have
\begin{multline}
\label{proposition hole filling equation 1}
\int_{B_R} \chi_\M \left( \frac1T \phi_T^2 +|\nabla\phi_T+\mathbf{e}|^2+1\right)
\lesssim
\left( \frac{R}{\sqrt{T}} \right)^{\epsilon d}
\int_{B_{\sqrt{T}}} \chi_\M \left( \frac1T \phi_T^2 +|\nabla\phi_T+\mathbf{e}|^2+1\right),
\end{multline}
where
\begin{equation}
\label{proposition hole filling equation 2}
\epsilon:= \frac{1}{d} \frac{\ln \frac1\theta}{\ln 3}, \qquad \theta:=\frac{C_0}{1+C_0}\,,
\end{equation}
and $C_0$ is a positive constant depending only on $d$ and the ellipticity constant.
\end{lemma}

\begin{remark}
Note that \eqref{proposition hole filling equation 1} immediately implies the same statement with $\phi_T$ replaced by $\phi_T^\ext$ in the LHS. We implicitly rely on this observation throughout the paper.
\end{remark}

\begin{proof}[Proof of Lemma~\ref{lemma hole filling}]
Let $\eta=\eta_R$ be a smooth cut-off of $B_{R+1}$ in $B_{2R}$ and put 
\[
c:=\frac{\int_{B_{2R}\setminus B_{R+1}}\eta^2\chi_\M\phi_T}{\int_{B_{2R}\setminus B_{R+1}}\eta^2\chi_\M}. 
\]
Testing \eqref{massive corrector} with $\eta^2 (\phi_T-c)$ one obtains
\begin{multline}
\label{proof caccioppoli equation 1}
0=\frac1T \int_{\R^d} \chi_\M \eta^2 \phi_T(\phi_T-c)+\int_{\R^d} a( \nabla\phi_T+\mathbf{e})\nabla(\eta^2(\phi_T-c))
\\
=
\frac1T \int_{\R^d} \chi_\M \eta^2 \phi_T^2 +\int_{\R^d} \eta^2 a( \nabla\phi_T+\mathbf{e})\cdot(\nabla\phi_T+\mathbf{e})
\\
-\frac{c}T \int_{\R^d} \chi_\M \eta^2 \phi_T
+\int_{\R^d}2\eta(\phi_T-c) a ( \nabla\phi_T+\mathbf{e}) \cdot \nabla \eta
-\int_{\R^d} \eta^2 a( \nabla\phi_T+\mathbf{e})\cdot \mathbf{e}.
\end{multline}
The latter implies, via Young's inequality,
\begin{multline}
\label{proof caccioppoli equation 2}
\int_{B_{2R}} \chi_\M \eta^2 \left( \frac1T \phi_T^2 +|\nabla\phi_T+\mathbf{e}|^2\right)
\\
\lesssim
\frac1T \int_{B_{2R}} \chi_\M \eta^2|c|^2 + R^d
+
\int_{B_{2R}\setminus B_{R+1}}\chi_\M |\nabla \eta|^2 |\phi_T-c|^2
\\
\lesssim
R^d\left(\frac{|c|^2}T+1\right)
+
\frac1{R^2}
\int_{B_{2R}\setminus B_{R+1}}\chi_\M |\phi_T-c|^2
\\
\lesssim
R^d\left(\frac{|c|^2}T+1\right)
+
\int_{B_{2R}\setminus B_{R+1}} |\nabla\phi_T^\ext|^2\,.
\end{multline}
In the last step above we used Lemma~\ref{lemma about contradiction argument} (adapted to the annulus) and Poincar\'e's inequality.

By means of Jensen's inequality and the properties of the extension, the estimate \eqref{proof caccioppoli equation 2} yields
\begin{multline}
\label{caccioppoli equation 2}
\int_{B_R} \chi_M \left( \frac1T \phi_T^2 +|\nabla\phi_T+\mathbf{e}|^2+1\right)
\\
\le C_0
\int_{B_{2R+1}\setminus B_R}\chi_M \left( \frac1T \phi_T^2 +|\nabla\phi_T+\mathbf{e}|^2+1\right)\,.
\end{multline}
Note that in the RHS of \eqref{caccioppoli equation 2} we slightly enlarged the domain of integration: this accounts for the extension domains of the inclusions intersecting the boundary of  $B_{2R}\setminus B_{R+1}$ when estimating $\nabla\phi_T^\ext$ by $\chi_\M\nabla\phi_T$.

Formula \eqref{caccioppoli equation 2} in turn implies
\begin{multline}
\label{hole filling equation 1}
\int_{B_R} \chi_\M \left( \frac1T \phi_T^2 +|\nabla\phi_T+\mathbf{e}|^2+1\right)
\\
\le \theta
\int_{B_{3R}} \chi_\M \left( \frac1T \phi_T^2 +|\nabla\phi_T+\mathbf{e}|^2+1\right)
\end{multline}
with $\theta=\frac{C_0}{1+C_0}$.

By iterating \eqref{hole filling equation 1} we arrive at \eqref{proposition hole filling equation 1}.
\end{proof}



\end{appendices}

\end{document}